\documentclass[letterpaper, 11pt,  reqno]{amsart}

\usepackage{amsmath,amssymb,amscd,amsthm,amsxtra, esint, bbm, enumerate, relsize}
\usepackage{mathtools} % For floor and celing functions
\usepackage{skak, bbm} % chess pieces
\usepackage{ulem}
\usepackage[implicit=true]{hyperref}

\usepackage{mparhack}

\allowdisplaybreaks[2]

\usepackage{color}

\definecolor{gr}{rgb}   {0.,   0.69,   0.23 }
\definecolor{bl}{rgb}   {0.,   0.5,   1. }
\definecolor{mg}{rgb}   {0.85,  0.,    0.85}
\definecolor{yl}{rgb}   {0.8,  0.7,   0.}
\makeatletter
\@namedef{subjclassname@2020}{%
  \textup{2020} Mathematics Subject Classification}
\makeatother

\newtheorem{theorem}{Theorem} [section]

\newtheorem{lemma}[theorem]{Lemma}
\newtheorem{proposition}[theorem]{Proposition}
\newtheorem{remark}[theorem]{Remark}

\newtheorem*{ackno}{Acknowledgment}

\newcommand{\Z}{\mathbb{Z}}
\newcommand{\R}{\mathbb{R}}

\newcommand{\T}{\mathbb{T}}

\let\Re=\undefined\DeclareMathOperator*{\Re}{Re}
\let\Im=\undefined\DeclareMathOperator*{\Im}{Im}

\newcommand{\NB}{\mathbb{N}}

\newcommand{\FL}{\mathcal{F}L} %%%%%%%%%%%% Fourier-Lebesgue spaces

\newcommand{\al}{\alpha}
\newcommand{\be}{\beta}
\newcommand{\dl}{\delta}

\newcommand{\eps}{\varepsilon}
\newcommand{\kk}{\kappa} 
\newcommand{\g}{\gamma}

\newcommand{\ld}{\lambda}

\newcommand{\s}{\sigma}

\newcommand{\ft}{\widehat}

\newcommand{\wt}{\widetilde}
\newcommand{\cj}{\overline}
\newcommand{\dx}{\partial_x}
\newcommand{\dt}{\partial_t}
\newcommand{\dd}{\partial}

\newcommand{\ta}{\theta}

\newcommand{\tr}{\text{tr}}
\newcommand{\HS}{\mathfrak{I}_{2}}

\DeclarePairedDelimiter\ceil{\lceil}{\rceil}
\DeclarePairedDelimiter\floor{\lfloor}{\rfloor}

\renewcommand{\l}{\ell}

\newcommand{\les}{\lesssim}
\newcommand{\ges}{\gtrsim}

\newcommand{\vk}{\varkappa}

\newcommand{\jb}[1]
{\langle #1 \rangle}

\newcommand{\ind}{\mathbbm{1}}

\DeclareMathOperator{\Id}{Id}

\numberwithin{equation}{section}
\numberwithin{theorem}{section}

\usepackage{marginnote}

\begin{document}
\baselineskip = 12.7pt

\title[mKdV in $L^2$]
{A remark on the well-posedness of the modified KdV equation in $L^2$
}

\author[J.~Forlano]
{Justin Forlano }

\address{
Justin Forlano\\
School of Mathematics\\
The University of Edinburgh\\
and The Maxwell Institute for the Mathematical Sciences\\
James Clerk Maxwell Building\\
The King's Buildings\\
Peter Guthrie Tait Road\\
Edinburgh\\ 
EH9 3FD\\
 United Kingdom\\
 and\\
Department of Mathematics\\
University of California\\
Los Angeles\\
CA 90095\\
USA}

\email{j.forlano@ed.ac.uk}



\begin{abstract}
We study the real-valued modified KdV equation on the real line and the circle, in both the focusing and the defocusing case. 
By employing the method of commuting flows introduced by Killip and Vi\c{s}an (2019), we prove global well-posedness in $H^{s}$ for $0\leq s<\tfrac{1}{2}$. 
On the line, we show how the arguments in the recent paper by Harrop-Griffiths, Killip, and Vi\c{s}an (2020) may be simplified in the higher regularity regime $s\geq 0$.
 On the circle, we provide an alternative proof of the sharp global well-posedness in $L^2$ due to Kappeler and Topalov (2005), and also extend this to the large-data focusing case. 
\end{abstract}

%\date{\today}

%\vspace*{-5mm}

%%
%
\maketitle
%
%\tableofcontents

\vspace*{-6mm}

\section{Introduction}

We consider the real-valued modified Korteweg-de Vries equation (mKdV):
\begin{equation}\label{mkdv}
\dt q=-q'''+6\mu q^2  q',
\end{equation}
 posed on either the real line $\R$ or the circle $\T=\R /\Z$. We say that \eqref{mkdv} is defocusing if $\mu =+1$ and focusing if $\mu=-1$. The equation \eqref{mkdv} is also known as the Miura mKdV equation after \cite{Miura}, and this naming differentiates it from its generalisation the complex-valued (Hirota) mKdV equation~\eqref{Hirotamkdv}.

The mKdV equation~\eqref{mkdv} has attracted much attention from both applied and theoretical perspectives, in part
since it is both Hamiltonian and completely integrable. Indeed, with respect to the Poisson structure
\begin{align}
\{ F,G\}:=\int \tfrac{\dl F}{\dl q} \big( \tfrac{\dl G}{\dl q}\big)' dx \label{Poisson}
\end{align}
 on Schwartz space $\mathcal{S}$, the mKdV equation~\eqref{mkdv} is the Hamiltonian flow of
\begin{align}
H_{\text{mKdV}}(q)=\frac{1}{2}\int (q')^2 +\mu q^4  dx. \label{Hmkdv}
\end{align} 
Recall that on $\T$, Schwartz space $\mathcal{S}$ corresponds to the space $C^{\infty}(\T)$, while on $\R$, it consists of those infinitely differentiable functions which, with all of their derivatives, decay faster than any polynomial as $|x|\to \infty$. 
Some important consequences of the completely integrable structure of mKdV~\eqref{mkdv} are the existence of a Lax pair and an infinite number of conservation laws. In particular, the simplest conservation law here is the mass $M(q)=\frac{1}{2}\int q^2 dx $. Under the Poisson structure \eqref{Poisson}, $M$ generates translations and thus may be also thought of as the momentum.

The optimal well-posedness of mKdV~\eqref{mkdv} within the $L^2$-based Sobolev spaces $H^{s}$, $s\in \R$, has been a long studied question.
On the real line, mKdV~\eqref{mkdv} has the scaling symmetry $q_{\ld}(t,x):=\ld q(\ld^{3}t,\ld x)$, where $\ld >0$, and this induces the scaling-critical Sobolev exponent $s_{\text{crit}}=-\tfrac{1}{2}$, above which we expect well-posedness, and below which we expect ill-posedness.  
Whilst this scaling is no longer available on the circle, the heuristic provided by $s_{\text{crit}}$ is believed to persist in this setting. However, due to the different ways that dispersion manifests, the well-posedness theory diverges depending on the underlying geometry.

On the real line, it has long been known that Schwartz initial data leads to global Schwartz solutions \cite{Kato1, MTsut}. Well-posedness, without making use of dispersion, was obtained in $H^{s}(\R)$, for $s>\tfrac{3}{2}$, either by viewing mKdV~\eqref{mkdv} from the perspective of quasilinear hyperbolic equations~\cite{Kato2} or by energy methods \cite{ABFS, BS1}. This was improved by Kenig, Ponce and Vega \cite{KPV1, KPV2}, who exploited dispersion through the use of local smoothing estimates, and proved local well-posedness in $H^{s}(\R)$, for $s\geq \tfrac{1}{4}$. An alternative proof of this result in $H^{1/4}(\R)$ was given by Tao~\cite{Tao} using the Fourier restriction norm method. Global well-posedness for $s>\tfrac{1}{4}$ was established using the $I$-method by Colliander, Keel, Staffilani, Takaoka, and Tao~\cite{CKSTT1}, and at the end-point $s=\tfrac{1}{4}$ in \cite{Kishimoto, Guo}. We also mention the unconditional uniqueness results in \cite{MPV1,KOY}. 

It turns out that $H^{1/4}(\R)$ is the threshold for the local uniform continuous dependence of the solution map for \eqref{mkdv} with respect to the initial data \cite{KPV3, CCT1}. Consequently, it is impossible to construct solutions using the contraction mapping theorem when $s<\tfrac{1}{4}$. Using the short-time Fourier restriction norm method, Christ, Holmer, and Tataru~\cite{CHT} proved global existence of solutions for $-\tfrac{1}{8}<s<\tfrac{1}{4}$, with the uniqueness of these solutions unknown at the time; see also \cite{MPV1}.
Using the complete integrability of mKdV~\eqref{mkdv} (or more accurately, the complex-valued mKdV~\eqref{Hirotamkdv}), Koch and Tataru~\cite{KT}, and Killip, Vi\c{s}an, and Zhang~\cite{KVZ} obtained a priori bounds in the whole sub-critical range $s>-\tfrac{1}{2}$.

Recently, Harrop-Griffiths, Killip, and Vi\c{s}an~\cite{HGKV} completed the well-posedness study on the line by proving global well-posedness of mKdV~\eqref{mkdv} in the remaining range $-\tfrac{1}{2}<s<\tfrac{1}{4}$. 
 Their notion of solution is that the solution map has a unique continuous extension from Schwartz space to $H^{s}(\R)$. Note that uniqueness of such solutions is built into this definition although the solutions need not be distributional. 
Their result went far beyond this though: they proved global well-posedness in the same range for the complex-valued mKdV equation, and resolved, in the positive, the long standing open question of well-posedness of the cubic nonlinear Schr\"{o}dinger equation (NLS) on $\R$ for $-\tfrac{1}{2}<s<0$. Moreover, because of their local smoothing estimates, their solutions are in fact also distributional solutions when $s\geq -\frac 16$ in the complex-valued case and for $s>-\frac 12$ for \eqref{mkdv} on $\R$.

For the complex-valued mKdV and cubic NLS, these results are sharp as there is ill-posedness, in the sense of norm inflation, at the critical regularity $s_{\text{crit}}=-\tfrac{1}{2}$; see \cite{CCT2, Kishimoto2, Oh1, HGKV}. 
Interestingly, it is not known if their result is sharp for the defocusing real-valued mKdV~\eqref{mkdv} as there is no known ill-posedness at this time at the end-point $s_{\text{crit}}=-\tfrac{1}{2}$, although see \cite[Proposition A.3]{HGKV} for an ill-posedness result in the focusing case.

The approach that we will use in this paper is the method of commuting flows, which was introduced by Killip and Vi\c{s}an~\cite{KV}, and lead to the complete resolution of the well-posedness of KdV~\cite{KV} in the $L^2$-based Sobolev scale on the line. Their approach also obtained optimal well-posedness on the circle, providing an alternative to the argument by Kappeler and Topalov~\cite{KapTopKdv}. Broadly speaking, it utilises the completely integrable structure of the equation and proceeds by approximating the flow of the original equation by another flow which (i) Poisson commutes, (ii) is easily-solved in low regularity and (iii) converges, in some sense, to the original flow. 

The method of commuting flows as in \cite{KV} applies equally well on the line or the circle as it does not take into account the dispersive nature of the equation. 
Subsequently, this method was extended in the line setting in~\cite{BKV, HGKV} to incorporate dispersion through local-smoothing estimates, which led to impressive results regarding the well-posedness of the fifth-order KdV, and the cubic NLS and complex-valued mKdV equations on $\R$ (as discussed above), respectively.
  We also mention the very recent breakthroughs~\cite{KNV, HGKV2, HGKNV}, using these methods and more, to the low-regularity and large data well-posedness of the derivative NLS.

One of the goals of this paper is to show how the arguments in \cite{HGKV} can be simplified in the more regular setting with $0\leq s<\tfrac{1}{4}$ and still yield global well-posedness for mKdV~\eqref{mkdv}. Indeed, we show that in this regime, it is enough to argue as in \cite{KV}, without exploiting dispersion explicitly; see Theorem~\ref{THM:GWP} below. We hope that this paper may therefore partly serve as a bridge for a reader familiar with \cite{KV} to study \cite{HGKV}.

We now discuss the current well-posedness theory for the mKdV equation~\eqref{mkdv} on the circle. The works~\cite{ABFS, BS1} also apply on the circle yielding well-posedness in $H^{s}(\T)$, for $s>\tfrac{3}{2}$. By introducing the Fourier restriction norm method, Bourgain~\cite{BO2} proved local well-posedness in $H^{s}(\T)$, $s\geq \tfrac{1}{2}$, which was extended to global well-posedness using the $I$-method in~\cite{CKSTT1}. In stark contrast to the setting of the real line, the solution map on $\T$ fails to be locally uniformly continuous below $H^{\frac{1}{2}}(\T)$ \cite{CCT2}; see also \cite{BO2}. Refined energy based methods were applied in \cite{TT, NTT, MPV2}, culminating in local well-posedness and unconditional uniqueness in $H^{s}(\T)$ for $s\geq \tfrac{1}{3}$.

Combining their well-posedness argument for KdV in $H^{-1}(\T)$ from \cite{KapTopKdv} and the Miura transform, Kappeler and Topalov~\cite{KapTop} proved that the real-valued defocusing mKdV~\eqref{mkdv} is globally well-posed in $L^2(\T)$, in the sense that the solution map extends uniquely from $\mathcal{S}$ to $L^2(\T)$. 
In \cite{KapMol}, Kappeler and Molnar also obtained small data global well-posedness in $L^2(\T)$ for the focusing case.\footnote{This was a consequence of their main arguments, see \cite[Remark on pp. 2217]{KapMol}; see also Remark~\ref{RMK:FLP}.}
This result is sharp, as Molinet~\cite{Molinet} proved that the mKdV equation~\eqref{mkdv} is ill-posed below $L^2(\T)$. Moreover, Molinet used the short-time Fourier restriction norm method and also proved global existence of distributional solutions, without uniqueness, for both the focusing and the defocusing equations. This result also implies that the solutions constructed by Kappeler and Topalov are actually distributional solutions. We also mention that Schippa~\cite{Schippa} extended Molinet's result to $H^{s}(\T)$ for $s>0$.

We now move onto the main result of this paper.

\begin{theorem}\label{THM:GWP}
Let $0\leq s<\frac{1}{2}$ and consider either $\R$ or $\T$. Then, the real-valued mKdV~\eqref{mkdv} equation is globally well-posed for all initial data in $H^{s}$ in the sense that the solution map $\Phi$ extends uniquely from $\mathcal{S}$ to a jointly continuous map $\Phi:\R \times H^{s}\to H^{s}$.
\end{theorem}

In view of the above discussion, we believe that, whilst largely not new, Theorem~\ref{THM:GWP} and its proof have a number of benefits.
 Firstly, by assuming higher regularity, we may drastically simplify the argument in \cite{HGKV} for the real line while still obtaining well-posedness that, prior to \cite{HGKV}, was unknown in the range $0\leq s<\tfrac{1}{4}$. In particular, we show how the argument in \cite{KV} can be adapted to the mKdV equation. As we do not need to exploit local smoothing, this allows us to treat the line and circle settings in a parallel fashion, and we show how the structures in \cite{HGKV} can be adapted to the circle.
Moreover, Theorem~\ref{THM:GWP} on $\T$ provides an alternative argument to the global well-posedness of the defocusing mKdV in $L^2(\T)$ by Kappeler and Topalov~\cite{KapTop}. In particular, we do not use the Miura transform and thus are able to obtain large data well-posedness results for the focusing case.
Moreover, the solutions we construct agree with those of Kappeler and Topalov. By Sobolev embedding, our solutions are distributional solutions in $H^{s}$, for $s\geq \tfrac{1}{6}$; see Remark~\ref{RMK:H16}. They are also distributional solutions in the remaining range $0\leq s<\tfrac{1}{6}$, leaning on the works \cite{HGKV} on $\R$  or \cite{Molinet, Schippa} on $\T$.

As discussed, our proof of Theorem~\ref{THM:GWP} uses the method of commuting flows as in \cite{KV}. This approach is motivated by the fact that mKdV~\eqref{mkdv} admits a Lax pair, where the Lax operator we use at spectral parameter $\kk\geq 1$ is
\begin{align}
L_{q}(\kk)=\begin{bmatrix}
\kk-\dd& 0\\
0 &  \kk+\dd \\
\end{bmatrix}+Q \quad \text{where} \quad Q=\begin{bmatrix}
0 & q \\
-\mu q & 0
\end{bmatrix}. \label{lax}
\end{align}
 The Lax pair suggests that formally the quantity 
\begin{align}
\log \det (1+ \mu (\kk-\dd)^{-1}q (\kk+\dd)^{-1}q) \label{logdet}
\end{align}
will be conserved under the flow of \eqref{mkdv}. For $q\in L^2$, $(\kk-\dd)^{-1}q (\kk+\dd)^{-1}q$ is a trace-class operator so that the functional determinant in \eqref{logdet} is well-defined.
Unfortunately, the functional determinant may vanish, causing the logarithm of it in \eqref{logdet} to become ill-defined.

Following \cite{KVZ}, we instead consider the quantity 
\begin{align}
A(\kk,q):=\mu\,  \l(\kk) \sum_{m=1}^{\infty} \tfrac{(-\mu)^{m-1}}{m}\tr\big\{ \big[(\kk-\dd)^{-1}q(\kk+\dd)^{-1}q \big]^{m}\big\}, \label{Alogdet}
\end{align}
where $\l(\kk) =1$ on $\R$ and $\l(\kk)=\tanh(\kk/2)$ on $\T$.
This quantity formally represents a series expansion of \eqref{logdet}.  For real-valued $q\in L^2$, modulo the constant $\l(\kk)$, this is exactly the quantity considered in \cite{KVZ, HGKV}, which can be seen by cycling the trace. As mentioned, while the logarithm in \eqref{logdet} may fail to be defined, the series \eqref{Alogdet} converges absolutely as long as $\kk^{-\frac{1}{2}} \|q\|_{L^2}$ is small enough and it is conserved under the flow of mKdV~\eqref{mkdv}. This smallness condition can be ensured either by scaling on the real line or by choosing $\kk$ sufficiently large when on the circle.

The quantity $A(\kk,q)$ is a generating function (in inverse powers of $\kk$) for the conservation laws of \eqref{mkdv}; see \eqref{Aexpand}. We prove these statements about $A(\kk,q)$ in Section~\ref{SEC:dynamicsR}.  The expansion of $A(\kk,q)$ motivates the definition of the approximating flow $H_{\kk}$ in \eqref{Hkk}, which we show is globally well-posed in $L^2$ for small enough data or large enough $\kk$. See the beginning of Section~\ref{SEC:Hk}. In order to understand the $H_{\kk}$-flow, we follow \cite{HGKV} in Section~\ref{SEC:Pre} and introduce some auxiliary functions based on the matrix Green's function of $L_{q}(\kk)$. In \cite{HGKV}, the off-diagonal matrix elements $g_{12}(\kk)$ and $g_{21}(\kk)$ were of great use.

In the real-valued setting, we find that it is convenient and natural to consider linear combinations of these and we introduce the new variables $g_{-}(\kk)$ and $g_{+}(\kk)$; see \eqref{grp}. Interestingly, whilst the map $q\mapsto g_{+}(\kk)$ is bounded from $L^2$ to $H^{1}$, the map $q\mapsto \tfrac{\mu}{4\kk}g_{-}(\kk)$ is a real-analytic diffeomorphism from $L^2$ to $H^{2}$, thus gaining an additional derivative over each of $g_{12}(\kk), g_{21}(\kk)$ and $g_{+}(\kk)$.
This gain of regularity is not available in the complex-valued setting; see Remark~\ref{RMK:Cr}. 

In Section~\ref{SEC:Hk}, we complete the well-posedness argument by following the approach in \cite{KV} and using the change of variables $q\mapsto \tfrac{\mu}{4\kk}g_{-}(\kk)$ in order to show that the $H_{\kk}$-flow well-approximates the $H_{\text{mKdV}}$-flow in $L^2$.

We conclude this introduction with a few remarks.

\begin{remark}\rm  \label{RMK:FLP}
Whilst the well-posedness of mKdV~\eqref{mkdv} in $L^2(\T)$ is sharp, the a priori bounds in Schwartz space in \cite{KVZ} suggest that some form of well-posedness may persist below $L^2(\T)$. In this setting, it is expected that one should instead consider the following renormalized mKdV equation:
\begin{align}
\dt q=-q'''+6\mu \bigg( q^2 -\int_{\T} q^2 dx \bigg)  q'. \label{renormmkdv}
\end{align}
This equation was first introduced by Bourgain~\cite{BO2}. In $L^2(\T)$,  mKdV~\eqref{mkdv} and the renormalized mKdV~\eqref{renormmkdv} are equivalent, as solutions of the former can be related to solutions of the latter via the gauge transformation 
\begin{align}
q(t,x) \mapsto q\bigg(t, x+6\mu t \int_{\T}q^2 dx  \bigg). \label{gauge}
\end{align} 
However, outside of $L^2(\T)$, we expect that \eqref{renormmkdv} may be well-posed even though \eqref{mkdv} is not.
In this direction, Kappeler and Molnar~\cite{KapMol} proved that the defocusing renormalized mKdV equation~\eqref{renormmkdv} is locally well-posed in the Fourier-Lebesgue spaces $\FL^{p}(\T)$ for $p>2$ and also globally well-posed for small data (the small data global result extends to the focusing case too). Note that for $p>2$, it holds that $L^2(\T) \subset \FL^p(\T)$, so they construct solutions outside of $L^2(\T)$. It would be of interest to understand how the gauge transformation \eqref{gauge} could be implemented within the method of commuting flows and, in particular, whether the arguments in this paper could be extended to the Fourier-Lebesgue setting on $\T$. 
\end{remark}

\begin{remark}\rm \label{RMK:Hirota}

On the real line, the proof of Theorem~\ref{THM:GWP} also applies to the complex-valued mKdV equation
\begin{align}
\dt q = -q'''+6\mu |q|^{2} q', \label{Hirotamkdv}
\end{align}
where $q(t):\R\to \mathbb{C}$.  
 In this setting, one no longer benefits from using the change of variables $\tfrac{\mu}{4\kk} g_{-}(\kk)$ as a map from $L^2$ to $H^{2}$, and instead one can proceed more closely to \cite{HGKV} and use a single component of the diagonal matrix Green's function, such as $g_{12}(\kk)$. 
 However, we stress that our result does not extend to the complex-valued mKdV equation on $\T$. This is why we chose to focus on the real-valued setting.
 
  To understand why this is the case, we recall that Chapouto~\cite{Chapouto1, Chapouto2} showed that \eqref{Hirotamkdv} and even its renormalized variant, in the sense of the corresponding gauge transformation to \eqref{gauge}, are ill-posed below $H^{1/2}(\T)$. The cause of the instability is the possibility that the momentum 
 \begin{align*}
P(q)=\text{Im} \int_{\T} \cj{q} q' dx
\end{align*}
may be infinite outside $H^{1/2}(\T)$. For smooth initial data, the momentum is a conserved quantity of \eqref{Hirotamkdv}. If additionally $q$ is real-valued, this conservation is trivial since $P(q)=0$. 

It was proposed instead in \cite{Chapouto1} that one should consider a second renormalised complex-valued mKdV equation for study below $H^{1/2}(\T)$, which is equivalent to \eqref{Hirotamkdv} for regular enough $q$ via a gauge transformation involving the momentum.
As discussed in Remark~\ref{RMK:FLP}, we do not know yet how to effectively implement such gauge transformations within the method of commuting flows.
\end{remark}

\section{The variables $\g(\kk)$, $g_{-}(\kk)$ and $g_{+}(\kk)$.}\label{SEC:Pre}

Our goal in this section is to understand the variables $\g(\kk)$, $g_{-}(\kk)$ and $g_{+}(\kk)$. We begin with some notation and preliminary results.
Our conventions for the Fourier transform on the line are 
\begin{align*}
\ft f(\xi)=\frac{1}{\sqrt{2\pi}}\int_{\R} e^{-i\xi x}f(x)dx \quad \text{and}\quad f(x)=\frac{1}{\sqrt{2\pi}}\int_{\R}e^{i\xi x}\ft f(\xi)d\xi.
\end{align*}
On the circle, we define the Fourier transform by 
\begin{align*}
\ft f(\xi)=\int_{0}^{1}e^{-ix\xi}f(x)dx, \quad \text{and} \quad f(x)=\sum_{\xi \in 2\pi \Z}\ft f(\xi)e^{ix\xi}.
\end{align*}
We denote by $\mathcal{S}$ the class of Schwartz functions on the line or the class of $C^{\infty}$ -functions if on the circle. 
For $\kk\geq 1$, we define the $L^2$-based Sobolev space $H^{s}_{\kk}$ via the norm 
\begin{align*}
\|f\|_{H^{s}_{\kk}} = \| (4\kk^2+\xi^2)^{\frac{s}{2}} \ft f(\xi)\|_{L^2}, 
\end{align*}
which we take with respect Lebesgue measure on $\R$ or counting measure on $2\pi \Z$.

The functional derivative $\dl / \dl q$ is defined as 
\begin{align*}
\tfrac{d}{ds}\big\vert_{s=0} F(q+sf) = dF\vert_{q}(f)=\int \tfrac{\dl F}{\dl q}(x) f(x)dx.
\end{align*}

For $0<\s<1$ and $\kk \geq 1$, we define the operator $(\kk \pm \dd)^{-\s}$ using the Fourier multiplier $(\kk \pm i\xi)^{-\s}$, where, for $\arg z\in (-\pi, \pi]$, we define 
\begin{align*}
z^{-\s} = |z|^{-\s} e^{-i \s \arg z}.
\end{align*}
This convention implies that 
\begin{align*}
[ (\kk \pm \dd)^{-\s}]^{\ast} = ( \kk\mp \dd)^{-\s}.
\end{align*}

In order to deal with both geometries in as parallel a fashion as possible, we consider the following class of potentials:
Given $\dl>0$, let 
\begin{align*}
B_{\dl, \kk} := \{ q\in L^2\,\, :\, \, \kk^{-\frac{1}{2}}\|q\|_{L^2 }  \leq \dl\},
\end{align*}
for $\kk\geq 1$.
We recall the definition of the Lax operator in \eqref{lax}, 
which we view as an unbounded operator on $L^2 \times L^2$ with domain $H^{1} \times H^{1}$.
The Green's function of $L_{q}(\kk)$ will be a matrix valued perturbation of the zero potential case, i.e. when $q=0$. In this case of zero potential, we have
\begin{align}
R_{0}(\kk):=L_{0}(\kk)^{-1}= 
\begin{bmatrix}
(\kk-\dd)^{-1} & 0\\
0 & & (\kk+\dd)^{-1}
\end{bmatrix}, \label{R0}
\end{align}
which has integral kernel
\begin{align}
G_{0}(x,y;\kk)&=e^{-\kk|x-y|} 
\begin{bmatrix}
\ind_{\{ x<y\}} & 0 \\
0 & \ind_{\{ y<x\}}
\end{bmatrix} \quad \text{on} \quad\R, \label{G0R} \\
G_{0}(x,y;\kk)&=\frac{1}{1-e^{-\kk}} 
\begin{bmatrix}
e^{\kk( x-y-\ceil{x-y})} & 0 \\
0 & e^{-\kk(x-y-\floor{x-y})}
\end{bmatrix} \quad \text{on} \quad\T, \label{G0T}
\end{align}
where $\ceil{x}$ and $\floor{x}$ denote the smallest integer $n$ such that $x\leq n$, or the largest integer $m$ such that $m\leq x$, respectively.
By formally iterating the resolvent identity, we arrive at an expression for the resolvent $R(\kk):=L_{q}(\kk)^{-1}$:
\begin{align}
R(\kk)=R_{0}(\kk)+\sum_{\l=1}^{\infty}(-1)^{\l} (R_{0}(\kk)Q)^{\l}R_{0}(\kk).  \label{Rseries}
\end{align}
The following proposition justifies that this series converges if $q\in B_{\kk,\dl}$, for $\dl>0$ sufficiently small and is an inverse to $L_{q}(\kk)$.

\begin{proposition}\label{PROP:Greens}
There exists $\dl>0$ such that for all $\kk\geq 1$, and $q\in  B_{\dl,\kk}$, the inverse $R(\kk)=L_{q}(\kk)^{-1}$ admits an integral kernel $G(x,y;\kk,q)$, for which the mapping 
\begin{align}
L^2\ni q\mapsto G-G_0 \in (H_{\kk}^{\frac{3}{4}-\eps}\times H_{\kk}^{\frac{3}{4}-\eps})   \otimes (H_{\kk}^{\frac{3}{4}-\eps}\times H_{\kk}^{\frac{3}{4}-\eps}) \label{Gmap}
\end{align}
is continuous, for any $\eps>0$. Moreover, $G-G_0$ is a continuous function of $(x,y)$. 
\end{proposition}

Before we give a proof of Proposition~\ref{PROP:Greens}, we state a useful lemma.

\begin{lemma}\label{LEM:HSq}
Let $\kk\geq 1$ and $0<\al<\tfrac{1}{4}$. Then, for any $f\in \mathcal{S}$, we have
\begin{align}
\| (\kk\pm \dd)^{-1} f(\kk \mp \dd)^{-1}\|_{\HS} & \les\kk^{-\frac{1}{2}} \|f\|_{H^{-1}_{\kk}}, \label{gammaHs} \\
\| (\kk\pm\dd)^{-\frac{1}{2}+\al}f (\kk\mp \dd)^{-\frac{1}{2}+\al}\|_{\HS} & \les  \kk^{-\frac{1}{2}-2\al}\|f\|_{L^2}, \label{HSq14} \\
\| f(\kk\pm \dd)^{-1}f\|_{\HS} & \les \|f\|_{L^2}. \label{HSq12}
\end{align}
\end{lemma}
\begin{proof} 
In the following, we consider the circle case, as the case of the real line is similar.
For \eqref{gammaHs}, we note that for any $h\in C^{\infty}(\T)$, we may write
\begin{align*}
(\kk \pm \dd)^{-1}f(\kk \mp\dd)^{-1}h = \int_{0}^{1} K_{\pm}(x,y)h(y)dy,
\end{align*}
where the kernel $K_{\pm}$ is given by 
\begin{align*}
K_{\pm}(x,y)= \sum_{\xi \in 2\pi \Z} \sum_{\eta \in 2\pi\Z} \frac{\ft f(\xi -\eta)}{(\kk \pm i\xi)(\kk \mp i(\xi-\eta))}e^{i\xi x}e^{-i\eta y}.
\end{align*}
Therefore, by Plancherel, 
\begin{align*}
\| (\kk\pm \dd)^{-1} f(\kk \mp \dd)^{-1}\|_{\HS(\T)}^{2} & = \iint |K_{\pm}(x,y)|^2 dxdy \\
&  =\sum_{\xi\in 2\pi \Z} |\ft f(\xi)|^2 \sum_{\eta \in 2\pi\Z} \tfrac{1}{(\kk^2+\eta^2)(\kk^2+(\xi+\eta)^2}.
\end{align*}
In the summation over $\eta$, we estimate separately the contributions due to $|\xi+\eta|\ll |\eta|$, $|\xi+\eta|\sim |\eta|$ and $|\xi+\eta|\gg |\eta|$ and obtain 
\begin{align}
\| (\kk\pm \dd)^{-1} f(\kk \mp \dd)^{-1}\|_{\HS(\T)}^{2}  \les \kk^{-1} \sum_{\xi\in 2\pi \Z} \frac{|\ft f(\xi)|^{2}}{ 4\kk^2+ \xi^2} \sim \kk^{-1} \|f\|_{H^{-1}_{\kk}}. \label{gammaHs1}
\end{align}
Similarly, the proof of \eqref{HSq14} reduces to showing 
\begin{align*}
\sum_{\eta \in 2\pi\Z} \tfrac{1}{(\kk^2+\eta^2)^{\frac{1}{2}-\al} (\kk^2 +(\xi+\eta)^2)^{\frac{1}{2}-\al}}  \les \kk^{-1+4\al},
\end{align*}
for any $\al <\tfrac{1}{4}$,
which can also be dealt with by considering the same contributions to the sum as in showing \eqref{gammaHs1}. The restriction $\al<\tfrac{1}{4}$ appears naturally as a summation condition in the case $|\xi+\eta|\sim |\eta|$.
The remaining estimate \eqref{HSq12} follows from the observation that the kernels for the operators $(\kk\pm\dd)^{-1}$ in \eqref{G0R} and \eqref{G0T} belong to $L^{\infty}(\T \times \T)$, uniformly in $\kk\geq 1$.
\end{proof}

\begin{proof}[Proof of Proposition~\ref{PROP:Greens}]
The series \eqref{Rseries} converges in operator norm uniformly for $\kk\geq 1$ and $q\in  B_{\dl,\kk}$ using \eqref{gammaHs} and \eqref{HSq12}. Moreover, this shows that $R(\kk)-R_{0}(\kk)\in \HS$ and hence admits an integral kernel $G(x,y;\kk)$ in $L^2$. Then, the integral kernel $G-G_0$ is a continuous function of $(x,y)$ as \eqref{HSq14} and \eqref{HSq12} imply that $R-R_{0}$ is Hilbert-Schmidt from $H^{-\al}\times H^{-\al}$ to $H^{\al}\times H^{\al}$ for any $\al<\tfrac{3}{4}$.
\end{proof}

On the line and the circle, Proposition~\ref{PROP:Greens} implies that $(G-G_0)(x,y)$ extends continuously to the diagonal $x=y$, and thus we may define:
\begin{align}
\begin{split}
\g(x;\kk)& := \l(\kk) \big[ (G-G_0)_{11}(x,x;\kk)+(G-G_0)_{22}(x,x;\kk)\big],\\
g_{-}(x;\kk)&: = \l(\kk) \big[ G_{21}(x,x;\kk)-\mu G_{12}(x,x;\kk)\big],\\
g_{+}(x;\kk)& := \l(\kk)\big[ G_{21}(x,x;\kk)+\mu G_{12}(x,x;\kk)\big],
\end{split}\label{grp}
\end{align}
where 
\begin{align}
\l(\kk)= 1 \quad \text{on} \quad \R, \quad \text{and} \quad \l(\kk)= \tanh(\kk/2)=\tfrac{1-e^{-\kk}}{1+e^{-\kk}} \quad \text{on} \quad \T. \label{lk}
\end{align}

\begin{proposition}[Properties of $\g, g_{+}$ and $g_{-}$]\label{PROP:gpr}
 There exists $\dl_0>0$ such that for all $0<\dl\leq \dl_0$ the following hold: given $\kk\geq 1$ and $q\in  B_{\dl,\kk}$, the maps $q\to \g(\kk)$ and $q\to g_{+}(\kk)$ are bounded from $B_{\dl,\kk}$ to $H^{1}_{\kk}$ and satisfy 
\begin{align}
\|\g(\kk)\|_{H^{1}_{\kk}}& \les \kk^{-\frac{1}{2}}\|q\|_{L^2}^{2}, \label{g1}\\
\|g_{+}(\kk)\|_{H^{1}_{\kk}}& \les \|q\|_{L^2} , \label{pbd}
\end{align}
and the map $q\to \tfrac{\mu}{4\kk}g_{-}(\kk)$ is a real-analytic diffeomorphism from $ B_{\dl,\kk}$ to $H^{2}_{\kk}$ satisfying 
\begin{align}
\big\| \tfrac{\mu}{4\kk}g_{-}(\kk)\big\|_{H^{2}_{\kk}}\les \|q\|_{L^2}. \label{rh2}
\end{align}
Furthermore, the following identities hold 
\begin{align}
\g'(\kk)& = 2q g_{+}(\kk), \label{Gderiv} \\
g_{+}'(\kk)& =-2\kk g_{-}(\kk)+2\mu q(\g(\kk)+1), \label{pderiv}\\
g_{-}'(\kk)&=-2\kk g_{+}(\kk). \label{rderiv}
\end{align}
in the sense of distributions.
Lastly, if $q\in \mathcal{S}\cap B_{\dl,\kk}$, then $\g(\kk), g_{+}(\kk)$ and $g_{-}(\kk)$ also belong to $\mathcal{S}$.
\end{proposition}
\begin{proof} 
For $f\in \mathcal{S}$, the expansion \eqref{Rseries} implies that
\begin{align}
\jb{f, \g(\kk)} &  =\l(\kk)\sum_{m=1}^{\infty} (-\mu)^{m} \Big[ \tr\{ (\kk-\dd)^{-1}f [(\kk-\dd)^{-1}q(\kk+\dd)^{-1}q]^{m}\} \label{gseries} \\
& \hphantom{XXXXXXXXX} + \tr\{ (\kk+\dd)^{-1}f [(\kk+\dd)^{-1}q(\kk-\dd)^{-1}q]^{m}\}\Big], \notag\\
\jb{f, g_{-}(\kk)}& = \l(\kk)\sum_{m=0}^{\infty} \mu (-\mu)^{m} \Big[ \tr\{ (\kk-\dd)^{-1}f (\kk+\dd)^{-1}q [(\kk-\dd)^{-1}q(\kk+\dd)^{-1}q]^{m}\}  \label{rseries}\\
& \hphantom{XXXXXXXXX} + \tr\{ (\kk+\dd)^{-1}f (\kk-\dd)^{-1}q [(\kk+\dd)^{-1}q(\kk-\dd)^{-1}q]^{m}\}\Big], \label{rseries2}
\end{align}
and $\jb{f, g_{+}(\kk)}$ satisfies a similar formula as $\jb{f, g_{-}(\kk)}$ does in \eqref{rseries} except the second term \eqref{rseries2} in the summation is subtracted instead of added. Thus, in order to verify \eqref{g1} and \eqref{pbd}, it suffices to individually estimate the summands in \eqref{gseries} and \eqref{rseries}. For \eqref{gseries}, when $m=1$, \eqref{gammaHs} and \eqref{HSq12} imply
\begin{align}
\begin{split}
&| \tr\{ (\kk-\dd)^{-1}f(\kk-\dd)^{-1} q(\kk+\dd)^{-1}q\}|   \\
&\hphantom{XX} \leq \| (\kk-\dd)^{-1}f(\kk+\dd)^{-1}\|_{\HS} \| (\kk+\dd)(\kk-\dd)^{-1}\|_{L^2 \to L^2} \|q(\kk+\dd)^{-1}q\|_{\HS}\\ & \hphantom{XX} \les \kk^{-\frac{1}{2}} \|q\|_{L^2}^{2} \|f\|_{H^{-1}_{\kk}}.
\end{split} \label{g1esti}
\end{align}
For $m\geq 2$, we proceed similarly along with \eqref{HSq12} to obtain
\begin{align}
|\tr\{ (\kk-\dd)^{-1}f [(\kk-\dd)^{-1}q(\kk+\dd)^{-1}q]^{m}\} | \les \kk^{-\frac{1}{2}}\|q\|_{L^2}^2 (\kk^{-\frac{1}{2}}\|q\|_{L^2})^{2(m-1)} \|f\|_{H^{-1}_{\kk}}. \label{gtrbd}
\end{align}
Summing this estimate in $m$ with \eqref{g1esti} as in \eqref{gseries} and using duality proves \eqref{g1}, provided that $\dl>0$ is sufficiently small.

We move onto verifying \eqref{pbd}. Similar to the above computations, for any $m\geq 1$, \eqref{gammaHs} and \eqref{HSq12} imply
\begin{align}
&|\tr\{ (\kk-\dd)^{-1}f (\kk+\dd)^{-1}q [(\kk-\dd)^{-1}q(\kk+\dd)^{-1}q]^{m}\} |  \notag\\
& \les  \kk^{-\frac{1}{2}} \|f\|_{H^{-1}_{\kk}} \|q (\kk-\dd)^{-1}q\|_{\HS}   \|(\kk\pm\dd)^{-1}q\|_{L^2 \to L^2}^{2(m-1)+1} \notag \\
& \les \kk^{-1} \|q\|_{L^2}^{3} (\kk^{-\frac{1}{2}}\|q\|_{L^2})^{2(m-1)} \|f\|_{H^{-1}_{\kk}}. \label{mboundp}
\end{align}
It remains to consider the case when $m=0$, which we treat by direct computation, focusing on the circle case. Using \eqref{G0T}, 
\begin{align*}
\l(\kk) \tr\{ (\kk+\dd)^{-1}f(\kk-\dd)^{-1} q\}  \\
& = \l(\kk)\int_{\T} f(x) \int_{\T}  (G_0)_{22}(x,y ;\kk)(G_{0})_{11}(y,x;\kk)q(y)dy dx \\
& = \l(\kk)\int_{\T} f(x) \int_{\T} (G_0)_{22}(x,y ;\kk)^{2} q(y)dy dx \\
& = \l(\kk)\tfrac{1+e^{-\kk}}{1-e^{-\kk}} \int_{\T} f(x) \int_{\T} (G_0)_{22}(x,y; 2\kk) q(y)dydx \\
& = \jb{ (2\kk+\dd)^{-1}f,q}_{L^2} \leq \|f\|_{H^{-1}_{\kk}} \|q\|_{L^2}.
\end{align*}
Summing \eqref{mboundp} over $m\geq 2$ completes the proof of \eqref{pbd}. We note that these arguments also imply that $g_{-}(\kk)\in H^{1}_{\kk}$. 
Assuming for now the veracity of \eqref{rderiv}, this regularity property of $g_{-}(\kk)$ is upgraded to $H^{2}_{\kk}$ by \eqref{rderiv} and \eqref{pbd}. This gives \eqref{rh2}. 

Now, \eqref{Gderiv} follows from considering $\jb{f,\g'(\kk)}$ and using \eqref{gseries} with the operator identities $f'=[\kk+\dd, f]=-[\kk-\dd,f]$, where $f\in \mathcal{S}$. For \eqref{pderiv}, the series representation similar to \eqref{rseries} and the identities 
\begin{align}
f'=-(\kk-\dd)f-f(\kk+\dd)+2\kk f=(\kk+\dd)f+f(\kk-\dd)-2\kk f \label{fcomm}
\end{align}
imply 
\begin{align}
\jb{f,g_{+}'(\kk)} =-2\kk\jb{f,g_{-}(\kk)}+2\mu \jb{f,q\g(\kk)} + 4\mu \kk \l(\kk) \tr\{ (\kk^2-\dd^2)^{-1}fq\}. \label{pderiv1}
\end{align}
Consider the last term on the right hand side of \eqref{pderiv1}.
The integral kernel for the operator $(\kk^2-\dd^2)^{-1}$ is 
\begin{align*}
K(x,y)=
\begin{cases}
\tfrac{1}{2\kk} e^{-\kk |x-y|} &\text{on} \quad  \R, \\
 \tfrac{1}{2\kk (1-e^{-\kk})}\big[ e^{-\kk \|x-y\|}+e^{\kk(\|x-y\|-1)}\big] & \text{on} \quad \T,
\end{cases}
\end{align*}
where $\|x-y\| :=\text{dist}(x-y,\Z)$.
See, for example, \cite{KVZ}. Therefore, on $\R$ 
\begin{align*}
2\kk\tr\{ (\kk^2-\dd^2)^{-1}fq\}=2\kk \int \tfrac{1}{2\kk}f(x)q(x)dx= \jb{f,q},
\end{align*}
while on $\T$, 
\begin{align*}
2\kk\tr\{ (\kk^2-\dd^2)^{-1}fq\}& = 2\kk \int \tfrac{1+e^{-\kk}}{2\kk(1-e^{-\kk})} f(x)q(x)dx = \tfrac{1}{\l(\kk)} \jb{f,q}.
\end{align*}
In either geometry, these identities inserted into \eqref{pderiv1} verify \eqref{pderiv}. Finally, \eqref{rderiv} follows in a similar way to \eqref{pderiv} using \eqref{fcomm}.

To show that $\g, g_{+}, g_{-}$ belong to $\mathcal{S}$ if $q\in \mathcal{S}$, it suffices to show that if $q\in H^{k}$, then $(\g,g_{+},g_{-})\in H^{k+1}\times H^{k+1}\times H^{k+2}$ for each $k\in \mathbb{N} \cup\{0\}$ and we have an estimate
\begin{align}
\| \g(\kk)\|_{H^{k+1}_{\kk}} +\|g_{+}(\kk)\|_{H^{k+1}_{\kk}} +\|g_{-}(\kk)\|_{H^{k+2}_{\kk}} \leq C_{k}( \|q\|_{H^{k}}), 
\label{smoothest}
\end{align}
where $C_{k}>0$ is a polynomial.
 We prove this by induction on $k$, where we have already proved the base case $k=0$. The inductive step follows from \eqref{Gderiv}, \eqref{pderiv} and \eqref{rderiv}.

 In order to establish that $\g, g_{+}$ and $g_{-}$ belong to $\mathcal{S}(\R)$ if $q\in \mathcal{S}(\R)\cap B_{\dl,\kk}$, we only need to show that  
\begin{align}
\|\jb{x}^{n}\g(\kk)\|_{L^2 (\R)} +\|\jb{x}^{n}g_{+}(\kk)\|_{H^{1}(\R)} +\|\jb{x}^{n} g_{-}(\kk)\|_{L^2 (\R)} \les_{n} \| \jb{x}^{n} q\|_{L^2 (\R)},\label{decays}
\end{align}
as the decay of all the derivatives will follow from \eqref{decays}, and \eqref{Gderiv}, \eqref{pderiv} and \eqref{rderiv}. To prove \eqref{decays}, we employ duality and the relations \eqref{gseries} and \eqref{rseries}. The main point is to commute the weight $x^{n}$ so that it lands on a nearby factor of $q$, instead of the test function $f$. This can be achieved by using the identity
\begin{align*}
x^{n} (\kk\pm \dd)^{-1} =\small{\sum_{m=0}^{n}} (\pm 1)^{m} \tfrac{n!}{(n-m)!} (\kk\pm \dd)^{-m-1}x^{n-m}.
\end{align*}

Finally, we verify the diffeomorphism claims for $\tfrac{\mu}{4\kk}g_{-}(\kk)$. We have from \eqref{rderiv} and \eqref{pderiv} that
\begin{align}
\tfrac{\mu}{4\kk}g_{-}(\kk)=(4\kk^2-\dd^2)^{-1}(q +\g(\kk)q). \label{rfull}
\end{align}
Thus, $d\big(\tfrac{\mu}{4\kk}g_{-}(\kk)\big)\big\vert_{q=0} =(4\kk^2-\dd^2)^{-1}$,
which is an isomorphism from $L^2$ into $H^{2}_{\kk}$. By the inverse function theorem, this alone ensures that $\tfrac{\mu}{4\kk}g_{-}(\kk)$ is a diffeomorphism from $B_{\wt{\dl},\kk}$ to $H^{2}_{\kk}$, where $\wt{\dl}>0$ is small and may depend on $\kk\geq 1$. To extend this diffeomorphism to the whole of $B_{\dl,\kk}$, we need a uniform in $\kk\geq 1$ estimate on the derivative $dr\vert_{q}$ throughout $B_{\dl,\kk}$. As part of this, we need the estimate
\begin{align}
\| d\g(\kk)\vert_{q} \|_{L^2 \to L^{\infty}} \les \kk^{-1} \|q\|_{L^2} \label{gderivbd}
\end{align}
for $q\in \mathcal{S}\cap B_{\dl,\kk}$ and reducing $\dl$, if necessary. This follows easily by arguments similar to verifying that $\g(\kk)\in H^{1}_{\kk}$ using \eqref{gseries}.
From \eqref{rfull}, it follows that for any test function $f$, we have 
\begin{align}
dr(\kk)\vert_{q}(f)= \tfrac{4\mu \kk}{4\kk^2-\dd^2}\big[ f(1+\g(\kk))+q d\g(\kk)\vert_{q}(f)\big]. \label{nonlinearr}
\end{align}
Therefore, assuming \eqref{gderivbd} and using \eqref{nonlinearr}, we have 
\begin{align*}
\tfrac{1}{4\kk}\| dr(\kk)\vert_{q =0} -dr(\kk)\vert_{q}\|_{L^2 \mapsto H^{2}_{\kk}} \les \|\g(\kk)\|_{L^{\infty}}+\|d\g(\kk)\vert_{q}\|_{L^2 \mapsto L^{\infty}} 
 \les  \kk^{-1}\|q\|_{L^2}^{2} \les \dl^{2},
\end{align*}
where the implicit constant is uniform in $\kk\geq 1$. Thus, $\frac{\mu}{4\kk} g_{-}(\kk):B_{\dl, \kk}\to H^{2}_{\kk}$ is a real-analytic diffeomorphism from $B_{\dl,\kk}$ to $H^2_{\kk}$. 
This completes the proof of the diffeomorphism claim and hence also the proof of Proposition~\ref{PROP:gpr}.
\end{proof}

It will be useful later on to separately consider the first few terms in the series expansions \eqref{gseries} and \eqref{rseries} for $\g$ and $g_{+}$. In particular, we write
$\g(\kk)=\g^{[2]}(\kk)+\g^{[\geq 4]}(\kk)$, where 
\begin{align}
\g^{[2]}(\kk)&=-2\mu (2\kk-\dd)^{-1} q \cdot  (2\kk+\dd)^{-1}q\quad \text{and} \quad \g^{[\geq 4]}(\kk):=\g(\kk)-\g^{[2]}(\kk). \label{g2term}
\end{align}
The superscript notation here is motivated by the fact that $\g^{[2]}$ represents the term in $\g(\vk)$ which is quadratic in $q$. The formula for $\g^{[2]}$ follows by computing explicitly the first term in the sum \eqref{gseries}. On the line, this reduces to computing a simple contour integral, while on the circle, we need the identity 
\begin{align*}
\sum_{\xi \in 2\pi \Z} \frac{1}{  [\kk \pm i\xi][\kk \pm i(\xi+\xi_1+\xi_2)][\kk \mp i(\xi+\xi_2)]  }= \frac{1+e^{-\kk}}{1-e^{-\kk}}\frac{1}{(2\kk\pm i\xi_1)(2\kk \mp i\xi_2)},
\end{align*}
which follows from partial fraction decompositions and the identity 
\begin{align*}
\sum_{\xi\in 2\pi \Z} \big(   \tfrac{1}{\kk+i\xi}+\tfrac{1}{\kk-i\xi}\big)=\tfrac{1+e^{-\kk}}{1-e^{-\kk}}.
\end{align*}
Similarly, we also decompose $g_{+}(\kk)$ as $g_{+}(\kk)=g_{+}^{[1]}(\kk)+g_{+}^{[3]}(\kk) +g_{+}^{[\geq 5]}(\kk)$
where
\begin{align}
g_{+}^{[1]}(\kk) = -2\mu \dd (4\kk^2-\dd^2)^{-1}q, \quad g_{+}^{[3]}(\kk)=-4\mu \dd(4\kk^2-\dd^{2})^{-1}\big( q \g^{[2]}(\kk)\big), \label{pterms}
\end{align}
and $g_{+}^{[\geq 5]}(\kk):=g_{+}(\kk)-g_{+}^{[1]}(\kk)-g_{+}^{[3]}(\kk)$. We have the following bounds on the remainder pieces $\g^{[\geq 4]}(\kk)$ and $g_{+}^{[\geq 5]}(\kk)$:
\begin{align}
\|\g^{[\geq 4]}(\kk)\|_{L^1} & \les \kk^{-3} \|q\|_{L^2}^{4} \quad \text{and} \quad \|g_{+}^{[\geq 5]}(\kk)\|_{H^{1}_{\kk}} \les \kk^{-2}\|q\|_{L^2}^{5}. \label{gpbig}
\end{align}
The estimate on $g_{+}^{[\geq 5]}$ is immediate from \eqref{mboundp} for $m\geq 2$ and the bound for $\g^{[\geq 4]}$ follows by duality, and hence \eqref{gseries} with $f\in L^{\infty}$, and \eqref{gtrbd} this time placing $(\kk\pm \dd)^{-1}f(\kk\mp \dd)^{-1}$ in operator norm.

\begin{remark}\rm \label{RMK:Cr}
The additional regularity property of $g_{-}$ does not immediately seem to be available in the complex-valued setting. 
In the complex-valued setting,
\begin{align*}
\mu \, g_{-}(\kk)=  (2\kk+\dd)^{-1}(q+q\g)+(2\kk-\dd)^{-1}( \cj{q}+\cj{q}\g).
\end{align*}
In particular, the first order term of this is
\begin{align*}
\mu \, g_{-}^{[1]}(\kk)=  4\kk (4\kk^{2}-\dd^{2})^{-1} \Re (q) -2i\, \dd (4\kk^2-\dd^2)^{-1}  \Im(q).
\end{align*}
Hence, if $\Im (q)\neq 0$, there does not appear to be an additional smoothing effect from considering any linear combination of diagonal Green's functions. 
\end{remark}

\section{Conservation laws, dynamics, and equicontinuity}\label{SEC:dynamicsR}

In this section, we prove the claims in the introduction regarding the quantity $A(\kk,q)$ in \eqref{logdet}. Namely, we show that is well-defined for $q\in B_{\dl,\kk}$ and that it commutes with $A(\vk,q)$, for any $\vk\geq 1$. We then use the conservation of $A(\kk)$ to derive a priori bounds and equicontinuity results for the flow of mKdV~\eqref{mkdv} in $H^{s}$ for $0\leq s<\tfrac{1}{2}$. We begin with properties of $A(\kk,q)$.

\begin{proposition}\label{PROP:Aprops}
There exists $\dl>0$, so that for all $\kk\geq 1$ and $q\in  B_{\dl,\kk}$, the following hold: the series \eqref{Alogdet} converges uniformly in $\kk\geq 1$, it is differentiable with derivative
\begin{align}
\tfrac{\dl A}{\dl q}&=g_{-}(\kk), \label{derivAq}
\end{align}
and for $q\in \mathcal{S} \cap B_{\dl,\kk}$ it has the asymptotic expansion 
\begin{align}
A(\kk;q)= \tfrac{\mu}{\kk}M(q)-\tfrac{\mu}{4\kk^3}H_{\textup{mKdV}}(q)+\mathcal{O}(\kk^{-5}), \label{Aexpand}
\end{align}
and for any $\vk\geq 1$, it Poisson commutes with $A(\vk)$, that is, $\{ A(\kk), A(\vk)\}=0.$
%\begin{align*}
%\{ A(\kk), A(\vk)\}=0.
%\end{align*}
\end{proposition}

\begin{proof}
The convergence of \eqref{Alogdet} readily follows from \eqref{HSq12}. 
We compute
\begin{align*}
\tfrac{d}{d\ta}A(\kk;q+\ta f)\big\vert_{\ta=0}& 
 = \mu \l(\kk) \sum_{m=0}^{\infty} (-\mu)^{m}\tr \bigg\{  \big[ (\kk-\dd)^{-1}q(\kk+\dd)^{-1}q\big]^{m} \big( (\kk-\dd)^{-1}q(\kk+\dd)^{-1}f \\
 & \hphantom{XXXXXXXXXXX} +  (\kk+\dd)^{-1}f(\kk-\dd)^{-1}q\big)\big\}  \bigg\} \\
 & =\jb{f, g_{-}(\kk)}
\end{align*}
by regrouping terms and cycling the trace with $(\kk\pm \dd)^{-1}f$, which yields \eqref{derivAq}. For \eqref{Aexpand}, since $A(\kk; 0)=0$, Fubini's theorem gives
\begin{align*}
A(\kk,q)=\int_{0}^{1} \tfrac{d}{d\ta}A(\kk, \ta q)d\ta = \int q(x) \int_{0}^{1} g_{-}(\kk,\ta q) d\ta dx.
\end{align*}
The identities \eqref{Gderiv}, \eqref{pderiv} and \eqref{rfull} imply
\begin{align}
\begin{split}
g_{-}^{[1]}(\kk) = \tfrac{\mu}{\kk}q+\tfrac{\mu}{4 \kk^3}q''+\mathcal{O}(\kk^{-5}), & \quad g_{-}^{[3]}(\kk) = -\tfrac{4}{(2\kk)^3}q^{3}+\mathcal{O}(\kk^{-5}),  \\
 g_{-}^{[5]}(\kk) & =\mathcal{O}(\kk^{-5}),
\end{split} \label{rparts} 
\end{align}
where $g_{-}(\kk)=g_{-}^{[1]}(\kk)+g_{-}^{[3]}(\kk)+g_{-}^{[\geq 5]}(\kk)$.
Then, 
\begin{align*}
A(\kk, q)&= \int \tfrac{\mu}{2\kk} q^{2}  + \tfrac{\mu}{(2\kk)^3}qq''-\tfrac{1}{(2\kk)^3}q^4 dx+\mathcal{O}(\kk^{-5})=\tfrac{\mu}{\kk}M(q)-\tfrac{\mu}{4\kk^3}H_{\textup{mKdV}}(q)+\mathcal{O}(\kk^{-5}).
\end{align*}
Finally, we verify the Poisson commutativity property. We assume $\kk \neq \vk$. To begin, we note that \eqref{Gderiv}, \eqref{pderiv} and \eqref{rderiv} imply the following identities for $q\in \mathcal{S}\cap B_{\dl,\kk}\cap B_{\dl,\vk}$,
\begin{align}
\big[ g_{+}(\kk)g_{-}(\vk)-g_{-}(\kk)g_{+}(\vk)   \big] &=\tfrac{1}{2(\kk-\vk)}\dx \big\{ g_{+}(\kk)g_{+}(\vk)-g_{-}(\kk)g_{-}(\vk)-\mu [\g(\kk)+1][\g(\vk)+1] \big\}, \label{ident1} \\
\big[  g_{+}(\kk)g_{-}(\vk)+g_{-}(\kk)g_{+}(\vk)\big]&=-\tfrac{1}{2(\kk+\vk)} \dx  \big\{ g_{+}(\kk)g_{+}(\vk)+g_{-}(\kk)g_{-}(\vk)-\mu [\g(\kk)+1][\g(\vk)+1] \big\}. \label{ident2}
\end{align}
In the following, we argue that the boundary term from integration by parts always vanishes. This is obvious by periodicity on $\T$, while on $\R$, \eqref{g1}, \eqref{pbd} and \eqref{rh2} imply that $\g(\kk), g_{+}(\kk)$ and $g_{-}(\kk)$ all vanish as $|x|\to \infty$. Then, by \eqref{derivAq} and \eqref{rderiv}, we have 
\begin{align}
\big\{A(\kk),A(\vk) \big\} = \int \tfrac{\dl A(\kk)}{\dl q} \dx\big( \tfrac{\dl A(\vk)}{\dl q} \big) dx &= \int g_{-}(\kk) g_{-}'(\vk) dx \notag \\
& = \int \big[ \kk g_{+}(\kk)g_{-}(\vk)-\vk g_{+}(\vk)g_{-}(\kk)\big]dx. \label{AAcons1}
\end{align}
It follows from \eqref{rderiv} and \eqref{ident2} that 
\begin{align*}
\int \kk g_{-}(\kk)g_{+}(\vk)+\vk g_{+}(\kk)g_{-}(\vk)dx=0.
\end{align*}
Therefore, using \eqref{ident1} and \eqref{ident2}, we have 
\begin{align*}
\eqref{AAcons1} & =\int \kk\big[ g_{+}(\kk)g_{-}(\vk)+g_{-}(\kk)g_{+}(\vk)\big]+\vk \big[ g_{+}(\kk)g_{-}(\vk)-g_{+}(\vk)g_{-}(\kk)\big]dx=0.
\end{align*}

\end{proof}

\begin{lemma}
There exists $\dl>0$ such that the following holds: given $\kk\geq 1$, $q(0)\in \mathcal{S}\cap B_{\dl,\kk}$,  let $q(t)\in \mathcal{S}$ denote the global-in-time solution to mKdV~\eqref{mkdv}. Then, the function $g_{-}(\kk,t)=g_{-}(\kk, q(t))$ satisfies 
\begin{align}
\tfrac{d}{dt}g_{-}(\kk,t)&=-g_{-}'''(\kk,t)+6\mu q(t)^2 g_{-}'(\kk,t), \label{rHmkdvflow}
\end{align}
\end{lemma}
\begin{proof}
Note that conservation of mass implies that $q(t) \in B_{\dl, \kk}$ for every $t\in \R$ so that $g_{-}(\kk,t)$ exists for every $t\in \R$ and belongs to $\mathcal{S}$ by Proposition~\ref{PROP:gpr}. As to the equation \eqref{rHmkdvflow}, we recall that \cite[Corollary 4.8]{HGKV} implies that for each $\vk\geq 1$ such that $\vk\neq \kk$,  $g_{-}(\vk,t)$ evolves under the $A(\kk)$ flow according to
\begin{align}
\tfrac{d}{dt}g_{-}(\vk)&=\tfrac{2\mu \kk \vk }{\vk^2-\kk^2} \big[ g_{+}(\vk)(\g(\kk)+1)-g_{+}(\kk)(\g(\vk)+1)     \big]. \label{rAvk}
\end{align}
This is a consequence of the Lax pair for the $A(\kk)$-flow discovered in \cite[Proposition 4.7]{HGKV}, which also holds in the real-valued setting and on both the line and the circle. Then, \eqref{rHmkdvflow} follows from \eqref{rAvk} and the asymptotic expansion \eqref{Aexpand}.
\end{proof}

Recall that a subset $Q$ of $H^s$ is equicontinuous if 
\begin{align}
\lim_{h\to 0} \sup_{q\in Q}\|q(x+h)-q(x)\|_{H^s}=0. \label{equicty}
\end{align}
Whilst the conservation of the $L^2$-norm for regular solutions to \eqref{mkdv} has long been known, in this section we demonstrate that bounded orbits of \eqref{mkdv} in $L^2$ are equicontinuous. On the circle, equicontinuity and uniform boundedness is equivalent to pre-compactness. 
The pre-compactness of solutions on the circle in $H^{s}(\T)$ for $s\geq \tfrac{1}{3}$ was shown in \cite{MPV2}.
In the line setting, this was also known from \cite[Proposition 8.1]{HGKV}. We provide an alternative argument using \eqref{Alogdet} and employing some ideas from \cite{KVZ}. Our main goal in this section is to prove the following:

\begin{proposition}[Boundedness and equicontinuity]\label{PROP:equicty}
Let $0\leq s<\frac{1}{2}$.
Given $\dl>0$, there exists $\dl_0>0$ so that for each $\kk\geq 1$, $q\in B_{\dl_{0},\kk} \cap \mathcal{S}$ and for any Hamiltonian flow on $\mathcal{S}$ which conserves $A(\kk)$, we have the a priori estimate
\begin{align}
\|q(t)\|_{H^{s}}\les C(\|q(0)\|_{H^{s}}). \label{Hsbd}
\end{align}
Moreover, if $Q\subseteq B_{\dl, \kk}\cap \mathcal{S}$ is equicontinuous in $H^{s}$, then, for any two Hamiltonians $H_1$ and $H_2$, which Poisson commute with $A(\vk)$ for all $\vk\geq 1$, the set
\begin{align}
Q_{\ast}=\bigg\{ e^{J\nabla(tH_{1}+\tau H_2)}q\, :\, q\in Q, \,\, t,\tau \in \R,\,\, \kk\geq 1  \bigg\}
\label{Qast}
\end{align}
is equicontinuous in $H^{s}$.
\end{proposition}

In order to prove Proposition~\ref{PROP:equicty}, we need to recall some preliminary results.
We note that the results in this subsection are valid in either geometry of $\R$ or $\T$, with similar proofs. Thus, for any integral on the frequency side, we take it to mean that the integral is with respect to Lebesgue measure on $\R$ on the line, or counting measure on $\Z$ if on the circle.

It is well known \cite{Pego} that equicontinuity in $H^{s}$ is equivalent to tightness on the Fourier side, and that equicontinuity is the key tool in order to upgrade strong convergence in a low regularity $H^{s}$-space to a higher regularity space. More precisely: 

\begin{lemma}\label{LEM:equiequiv}
Fix $-\infty <\s<s<\infty$. Then: \\
\textup{(i)} a non-empty bounded subset $Q$ of $H^{s}$ is equicontinuous in $H^s $ if and only if 
\begin{align}
\lim_{\kk \to 0} \sup_{q\in Q} \int_{|\xi|\geq \kk} \jb{\xi}^{2s}|\ft q(\xi)|^2 d\xi=0. \label{tightness}
\end{align}
\textup{(ii)} A sequence $\{q_n\}$ is convergent in $H^{s}$ if and only if it is convergent in $H^{\s}$ and
equicontinuous in $H^{s}$.
\end{lemma}

We use a number of equivalent ways to express equicontinuity in $H^{s}$. To this end, we first state a useful lemma.
\begin{lemma}\label{LEM:equiN}
Let $-1<s<1$ and consider the Fourier multiplier operator $w(-i\dd,\kk)$ with multiplier
\begin{align}
w(\xi;\kk)=\tfrac{\kk^2}{\xi^2+4\kk^2}-\tfrac{(\kk/2)^{2}}{\xi^2+\kk^2}=\tfrac{3\kk^2\xi^2}{4(\xi^2+\kk^2)(\xi^2+4\kk^2)}. \label{w}
\end{align} Then: \\
\textup{(i)} For any $f\in \mathcal{S}$ and uniformly in $\kk\geq 1$, we have
\begin{align}
 \sum_{ N\in 2^{\mathbb{N}}}(\kk N)^{2s} \jb{f,w(-i\dd, \kk N)f} \les \|f\|_{H^{s}}^{2}\les\|f\|_{H^{-1}}^{2}+\kk^{2}  \sum_{ N\in 2^{\mathbb{N}}}N^{2s} \jb{f,w(-i\dd, \kk N)f}. \label{wbelow}
\end{align}
\textup{(ii)}  a non-empty bounded subset $Q$ of $H^{s}$ is equicontinuous if and only if 
\begin{align}
\lim_{\kk\to \infty} \sup_{q\in Q} \sum_{ N\in 2^{\NB}}(\kk N)^{2s} \jb{q,w(-i\dd, \kk N)q}=0 \label{equicond}
\end{align} 
\end{lemma}
\begin{proof}
All of part (i) was shown in \cite[Lemma 3.5]{KVZ}. For (ii), we first suppose that \eqref{equicond} holds. Then
\begin{align*}
\| P_{ \geq \kk} q\|_{H^{s}}^{2} & \sim \sum_{N\in 2^{\NB}} \int_{|\xi|\sim \kk N} (\kk N)^{2s}|\ft q(\xi)|^{2}d\xi \les  \sum_{N\in 2^{\NB}} (\kk N)^{2s} \jb{ P_{\geq \kk} q, w(-i\dd,\kk N)P_{\geq \kk}q},
\end{align*}
where $P_{<\kk}$ is the sharp projection onto frequencies $\{ |\xi|<\kk\}$ and $P_{\geq \kk}:=\Id -P_{<\kk}$.
Hence, \eqref{equicond} implies $Q$ is equicontinuous in $H^s$. For the reverse direction, given $\vk \geq 1$, \eqref{wbelow} implies 
\begin{align*}
\sum_{ N\in 2^{\NB}}(\kk N)^{2s} \jb{P_{\geq \vk} q,w(-i\dd, \kk N)P_{\geq \vk}q} \les \|P_{\geq \vk}q\|_{H^{s}}^{2}.
\end{align*}
On the other hand, for $0\leq s<1$,
\begin{align*}
\sum_{ N\in 2^{\NB}}(\kk N)^{2s} \jb{P_{< \vk} q,w(-i\dd, \kk N)P_{< \vk}q} 
& \les  \tfrac{\vk^{2}}{\kk^{2-2s}}\bigg( \sum_{ N\in 2^{\NB}}N^{-(2-2s)}\bigg) \int_{|\xi|<\vk} |\ft q(\xi)|^{2}d\xi \\
& \les  \tfrac{\vk^{2}}{\kk^{2-2s}}\sup_{q\in Q}\|q\|_{H^{s}}^{2}
\end{align*}
while if $-1<s<0$, 
\begin{align*}
\sum_{ N\in 2^{\NB}} (\kk N)^{2s} \jb{P_{< \vk} q,w(-i\dd, \kk N)P_{< \vk}q} \les \tfrac{\vk^{2|s|}}{\kk^{2|s|}} \sup_{q\in Q}\|q\|_{H^{s}}^{2}
\end{align*}
In either case, there exist $\be_1(s), \be_2(s)>0$ such that 
\begin{align*}
\sum_{ N\in 2^{\NB}}( \kk N)^{2s} \jb{ q,w(-i\dx, \kk N)q} \les \tfrac{\vk^{\be_1(s)}}{\kk^{\be_2(s)}} \sup_{q\in Q}\|q\|_{H^{s}}^{2} + \|P_{\geq \vk}q\|_{H^{s}}^{2}.
\end{align*}
Using the equicontinuity of $Q$, we now obtain \eqref{equicond}.
\end{proof}

\begin{lemma}\label{LEM:equ2}
Let $0\leq s<1$. A non-empty bounded subset of $Q$ of $H^{s}$ is equicontinuous in $H^{s}$ if and only if \begin{align}
  \lim_{\kk\to \infty} \sup_{q\in Q} \int \tfrac{\xi^4}{(\xi^2+4\kk^2)^{2-s}}  |\ft q(\xi)|^{2}d\xi =0.   \label{equicond5}
\end{align} 
\end{lemma}
The proof of this lemma is quite standard with the forward direction following from low and high frequency decomposition as in the proof of (ii) in Lemma~\ref{LEM:equiN}, and the converse follows from the inequality $\tfrac{2\xi^2}{\xi^2+4\kk^2} \ges 1-\chi_{[-\kk,\kk]}(\xi)$.
We note that Lemma~\ref{LEM:equ2} remains true with the condition \eqref{equicond5} replaced by the condition
\begin{align}
\lim_{\kk\to \infty} \sup_{q\in Q}\int \tfrac{\xi^2}{(\xi^2+4\kk^2)^{1-s}}|\ft q(\xi)|^{2}d\xi =0. \label{tightness2}
\end{align}

\begin{proof}[Proof of Proposition~\ref{PROP:equicty}]
We let $\dl>0$ be small enough so that all prior results hold for any $q(0)\in B_{\dl_{0}, \kk} \cap \mathcal{S}$, where $\dl_0 =\frac{1}{2}\dl$.
We first show that $A(\vk)$ is conserved under the $H_{\text{mKdV}}$ flow and the mass $M$ flow. Consider the $M$ flow: by \eqref{Gderiv}, \eqref{pderiv} and \eqref{derivAq}, we have
\begin{align*}
\{ M, A(\vk)\}=\int qr'(\vk) dx= -\int 2\vk q g_{+}(\vk)dx=-\vk\int \g'(\vk)dx=0. 
\end{align*}
Using \eqref{derivAq}, \eqref{Gderiv}, \eqref{pderiv} and \eqref{rderiv}, we have  
\begin{align}
\{ H_{\text{mKdV}}, A(\vk)\} =\int (-q''+2\mu q^3)g_{-}'(\vk) dx  = -2\vk \int q' g_{+}'(\vk) dx -4\mu \vk \int q^3 g_{+}(\vk)dx. \label{HA}
\end{align}
Now we examine the first term. Using \eqref{pderiv}, \eqref{rderiv} and \eqref{Gderiv},
\begin{align*}
-2\vk \int q' g_{+}'(\vk) dx & = 4\vk^2 \int q' g_{-}(\vk)dx -4\mu \vk \int q' q(\g(\vk)+1)dx \\
& =- 4\vk^2 \int q g_{-}'(\vk)dx -2\mu \vk \int (q^2)' \g(\vk)dx  \\
& = 8\vk^3 \int qg_{+}(\vk) dx +2\mu \vk \int q^2 \g'(\vk)dx \\
& = 4 \vk^{3} \int \g'(\vk)dx  +4\mu \vk \int q^3 g_{+}(\vk)dx. 
\end{align*}
Returning this to \eqref{HA}, we see that $\{ H_{\text{mKdV}}, A(\vk)\}=0$.

To deduce the a priori bound and the equicontinuity claims, we exploit the conservation of $A(\kk)$ under any Hamiltonian flow which commutes with $A(\kk)$. Given $q\in \mathcal{S}$, we let $\vk_0= 100(1+\|q\|_{L^2}^{2}) $ and define
\begin{align*}
A^{[2]}(\vk,q):= \mu \l(\kk) \tr\{ (\vk-\dd)^{-1}q(\vk+\dd)^{-1}q\} \quad \text{and} \quad A^{[\geq 4]}(\vk,q):=A(\vk,q)-A^{[2]}(\vk,q),
\end{align*}
for any $\vk \geq \vk_0$. The existence of $A(\vk,q)$ follows from the choice of $\vk_0$ and Lemma~\ref{LEM:HSq}.
We may compute $A^{[2]}(\vk)$ explicitly, since
\begin{align*}
 \l(\vk)\tr\{ (\vk-\dd)^{-1}q(\vk+\dd)^{-1}q\}= \int \tfrac{|\ft q(\xi)|^{2}}{2\vk-i\xi}d\xi = \int \tfrac{2\vk |\ft q(\xi)|^{2}}{\xi^2 +4\vk^2}d\xi,
\end{align*}
where in the second equality we exploited the fact that $q$ is real-valued.
Therefore
\begin{align}
A^{[2]}(\vk,q)-\tfrac{1}{2} A^{[2]}(\tfrac{\vk}{2},q) =\tfrac{2\mu }{\vk} \jb{q, w(-i\dd,\vk)q}. \label{rho2w}
\end{align}
Recalling \eqref{Alogdet}, Lemma~\ref{LEM:HSq} implies
\begin{align}
|A^{[\geq 4]}(\vk,q)| & \les  \vk^{-2}\|q\|_{L^2}^{4}. \label{rho4w}
\end{align} 
Fix $\kk\geq 1$, let $q(0)\in B_{\dl, \kk}\cap \mathcal{S}$, $\vk\geq \vk_0$ and let $q(t)$ be the corresponding solutions in $\mathcal{S}$.
In view of Proposition~\ref{PROP:Aprops}, $A(\vk, q(t))$ is conserved and by commutativity of the mass $M$, we have that $q(t)\in B_{\dl,\kk}\cap \mathcal{S}$ for all $t\in \R$. Moreover, \eqref{rho4w} also holds uniformly in $t\in \R$. 

Now, the conservation of $A(\kk)$, \eqref{rho2w} and \eqref{rho4w} imply that for any $\vk\geq \vk_0$,
\begin{align}
\tfrac{2}{\vk}\jb{q(t), w(-i\dd,\vk)q(t)}  
 & \leq   |A^{[\geq 4]}(\vk,q(t))-\tfrac{1}{2}A^{[\geq 4]}(\tfrac{\vk}{2},q(t)) | \notag  \\
 &\hphantom{X}  +|A^{[2]}(\vk,q(0))-\tfrac{1}{2} A^{[2]}(\tfrac{\vk}{2},q(0))|+ \tfrac{2 }{\vk} \jb{q(0), w(-i\dd,\vk)q(0)}   \notag \\
 & \leq C\vk^{-2}\|q(0)\|_{L^2}^{4} +\tfrac{2}{\vk}\jb{q(0), w(-i\dd,\vk)q(0)}. \label{equi1}
\end{align}
Given $N\in 2^{\mathbb{N}}$, we replace $\vk$ by $\vk N$ in \eqref{equi1} and apply \eqref{wbelow} to obtain 
\begin{align*}
\|q(t)\|_{H^s}^{2}\les \|q(0)\|_{L^2}^{2}+\vk \|q(0)\|_{L^2}^{4}+\vk^{2-2s}\|q(0)\|_{H^{s}}^2,
\end{align*}
provided that $s<\frac{1}{2}$. Choosing $\vk=\vk_0$, we arrive at \eqref{Hsbd}.

For the equicontinuity claim, we return to \eqref{equi1} and replace $\vk$ by $\vk N$, multiply by $(\vk N)^{2s}$ and sum over $N\in 2^{\mathbb{N}}$ to obtain 
\begin{align*}
\sum_{ N\in 2^{\NB}}(\vk N)^{2s} \jb{ q(t),w(-i\dd, \vk N)q(t)} &\les \vk^{-1+2s}\|q(0)\|_{L^2}^{4} +  \sum_{ N\in 2^{\NB}}(\vk N)^{2s} \jb{ q(0),w(-i\dd, \vk N)q(0)}.
\end{align*}
Now, Lemma~\ref{LEM:equiN} yields the claimed equicontinuity.
\end{proof}

\section{Well-posedness in $H^{s}$, $s\geq 0$}\label{SEC:Hk}

The expansion \eqref{Aexpand} motivates us to define
\begin{align}
H_{\kk}(q):=4\kk^2 M(q)-4\kk^{3}\mu A(\kk, q). \label{Hkk}
\end{align}
 Under the Poisson bracket \eqref{Poisson},
we have 
\begin{align*}
\{ H_{\kk},H_{\text{mKdV}}\}=4\kk^{2} \{ M, H_{\text{mKdV}}\}-4\kk^{3}\mu \{ A(\kk), H_{\text{mKdV}}\}=0,
\end{align*}
since $M$ and $A(\kk)$ Poisson commute with $H_{\text{mKdV}}$. We use the $H_{\kk}$-flow to approximate the $H_{\text{mKdV}}$-flow, as heuristically, \eqref{Aexpand} implies 
$H_{\kk}=H_{\text{mKdV}}+\mathcal{O}(\kk^{-2})$ for large $\kk$.

\begin{proposition}[Well-posedness of the approximating flow] \label{PROP:GWPHk}
There exists $\dl>0$ so that for all $\kk\geq 1$, the Hamiltonian flow induced by $H_{\kk}$, namely
\begin{align}
\tfrac{d}{dt}q &=4\kk^{2}q'-4\mu \kk^{3} g_{-}'(\kk),
  \label{Hkflow}
\end{align}
is globally well-posed for initial data in $B_{\dl, \kk}$. These solutions conserve $A(\vk)$ for every $\vk\geq 1$ and whenever the initial data belongs to $\mathcal{S}$, then the solutions do too.
Moreover, under the $H_{\kk}$ flow, we have 
\begin{align}
\tfrac{d}{dt}g_{-}(\vk)&=4\kk^2 g_{-}'(\vk)+\tfrac{8\vk \kk^4}{\kk^2-\vk^2}\big[ g_{+}(\vk)(\g(\kk)+1)-g_{+}(\kk)(\g(\vk)+1)\big]. \label{rHk}
%\frac{d}{dt}g_{+}(\vk)& =4\kk^2 g_{+}'(\vk) -\frac{8 \kk^4 }{\vk^2-\kk^2} \big[ \vk g_{-}(\vk) (\g(\kk)+1)-\kk g_{-}(\kk) (\g(\vk)+1) \big]
\end{align}
\end{proposition}

\begin{proof} The equation \eqref{Hkflow} follows from \eqref{derivAq} which gives 
\begin{align*}
 \tfrac{\dl H_{\kk}}{\dl q}=4\kk^{2}q-4\kk^{3}\mu g_{-}(\kk).
\end{align*}
Then, local well-posedness of \eqref{Hkflow} follows by first using the convenient change of variables $(t,x)\to (t,x-4\kk^2 t)$ and
%\begin{align*}
%q(t,x)=q(0,x+4\kk^2 t) -4\mu \kk^{3}\int_{0}^{t} g_{-}'(x+4\kk^2(t-t'); \kk, q(t'))dt.
%\end{align*}
noting that, by the diffeomorphism property of $g_{-}(\kk)$, the nonlinearity is Lipschitz from $B_{\dl,\kk}$ into $H^{2}$ so that the Cauchy-Lipschitz theorem applies. The estimates \eqref{smoothest} and \eqref{decays} imply that the solutions belong to $\mathcal{S}$ if the initial data do.

Global well-posedness follows from the conservation of $A(\vk)$ under the $H_{\kk}$ flow with the specific a priori bounds from Proposition~\ref{PROP:equicty}.

The claim that the solutions conserve $A(\vk)$ follows from Proposition~\ref{PROP:Aprops}. Finally, the equation \eqref{rHk} follows from \eqref{rAvk} and that the evolution $g_{-}(\vk)$ under the $M$-flow is simply linear transport in $x$ at unit speed.
\end{proof}

The key stepping stone to the proof of Theorem~\ref{THM:GWP} is the following result which makes precise how we use the formal idea that $H_{\kk}$ approximates $H_{\text{mKdV}}$ for large $\kk$.

\begin{proposition}\label{PROP:diffflow}
Let $\dl>0$ be sufficiently small, $\kk\geq 1$, $\vk\geq 4$ and $Q\subseteq B_{\dl, \kk}\cap \mathcal{S}$ be equicontinuous in $L^2$. Then, 
\begin{align}
\lim_{\kk\to \infty}\sup_{q\in Q}\sup_{|t|\leq T}   \| g_{-}( \vk ; e^{tJ\nabla(H_{\textup{mKdV}}-H_{\kk})}q)-g_{-}(\vk; q)\|_{H^{2}}=0. \label{H2limit}
\end{align}
\end{proposition}

\begin{proof} 
Fix $\vk\geq 4$ and let $\kk\geq 2\vk$.
We verify the weaker statement
\begin{align}
\lim_{\kk\to \infty}\sup_{q\in Q}\sup_{|t|\leq T}   \| g_{-}( \vk ; e^{tJ\nabla(H_{\textup{mKdV}}-H_{\kk})}q)-g_{-}(\vk; q)\|_{H^{-2}}=0. \label{H-2limit}
\end{align}
To see that this suffices to prove \eqref{H2limit}, we note that from Proposition~\ref{PROP:equicty}, the set $Q_{\ast}$ defined in \eqref{Qast}, with $H_1=H_{\text{mKdV}}$ and $H_2=H_{\kk}$, is equicontinuous in $L^2$. Hence, the diffeomorphism property of $\frac{\mu}{4\vk}g_{-}(\vk)$ from Proposition~\ref{PROP:gpr} and the fact that for any $h\in \R$, $g_{-}(x+h;\vk, q(x))=g_{-}(x;\vk, q(x+h)),$
imply that the set 
\begin{align*}
E:= \big\{   \tfrac{\mu}{4\vk}g_{-}(\vk, e^{tJ\nabla(H_{\textup{mKdV}}-H_{\kk})}q) \, :  q\in Q, \,\,\, t\in\R\big\}
\end{align*}
is equicontinuous in $H^{2}$. Combining this fact with \eqref{H-2limit}, our desired \eqref{H2limit} follows from Lemma~\ref{LEM:equiequiv}. 

By the fundamental theorem of calculus, \eqref{H-2limit} follows from 
\begin{align}
\lim_{\kk\to \infty}\sup_{q\in Q_{\ast}}\sup_{|t|\leq T}   \big\|  \tfrac{d}{dt}g_{-}(\vk, q)  \big\|_{H^{-2}}=0. \label{H-2limit2}
\end{align}
Under the difference flow, \eqref{rHmkdvflow} and \eqref{rHk} imply
\begin{align}
\begin{split}
\tfrac{d}{dt}g_{-}(\vk)& =8\vk^3 g_{+}(\vk)-4\mu \vk q^{2}g_{+}(\vk)+4\mu \vk q'(\g(\vk)+1) \\
& \hphantom{XX}+8\kk^2 \vk g_{+}(\vk)-\tfrac{8\vk \kk^4}{\kk^2-\vk^2}\big[ g_{+}(\vk)(\g(\kk)+1)-g_{+}(\kk)(\g(\vk)+1)\big].
\end{split} \label{diffflowr}
\end{align}
Carefully rewriting the right hand side of \eqref{diffflowr}, we have
\begin{align}
\tfrac{d}{dt}g_{-}(\vk)= \sum_{j=1}^{8} \mathsf{err}_{j}, \label{rexpansion}
\end{align}
where 
\begin{align*}
\mathsf{err}_{1}& =-\tfrac{8\vk^5}{\kk^2- \vk^2} g_{+}(\vk), \qquad \qquad \mathsf{err}_{2} = -\tfrac{8\vk \kk^4}{\kk^2-\vk^2}g_{+}(\vk)\g^{[\geq 4]}(\kk),\qquad \mathsf{err}_{3} = \tfrac{4\mu \vk^3}{\kk^2-\vk^2}q^2 g_{+}(\vk)\\
&\hphantom{XXXXXX}\mathsf{err}_{4}=-\tfrac{4\mu \vk^3}{\kk^2-\vk^2}q' (\g(\vk)+1)\qquad
\mathsf{err}_{5} = \tfrac{8\mu \vk \kk^2}{\kk^2-\vk^2}q g_{+}(\vk)\cdot \dd^{2} (4\kk^2-\dd)^{-1}q  \\
   \mathsf{err}_{6}&= \tfrac{8\vk \kk^4}{\kk^2-\vk^2}(\g(\vk)+1)g_{+}^{[\geq 3]}(\kk) \qquad \qquad \mathsf{err}_{7} =-\tfrac{4\mu \vk \kk^2}{\kk^2-\vk^2}(\g(\vk)+1)\cdot  \dd^3 (4\kk^2-\dd^2)^{-1}q\\
&\hphantom{XXXXXX} \mathsf{err}_{8} = -\tfrac{4\mu \vk \kk^2}{\kk^2-\vk^2}g_{+}(\vk)\cdot \dd (2\kk+\dd)^{-1}q\cdot  \dd(2\kk+\dd)^{-1}q,
\end{align*}
and we used \eqref{g2term} to express $\mathsf{err}_{8}$.

We need to estimate each of these in $H^{-2}$. By Proposition~\ref{PROP:gpr}, the embedding $L^1\subset H^{-1}$, \eqref{pbd} and \eqref{gpbig}, we have
\begin{align*}
\| \mathsf{err}_{1}\|_{H^{-2}}&\les \kk^{-2} \|g_{+}(\vk)\|_{H^{-2}}\les \kk^{-2}\|q\|_{L^2}, \\
\| \mathsf{err}_{2}\|_{H^{-2}} &\les \kk^{2}\|g_{+}(\vk)\|_{H^{1}}\| \g^{[\geq 4]}(\kk)\|_{L^1} \les \kk^{-1}\|q\|_{L^2}^{5}, \\
\|\mathsf{err}_{3}\|_{H^{-2}}&\les \kk^{-2}\|q^2\|_{L^1}\|g_{+}(\vk)\|_{L^{\infty}} \les \kk^{-2}\|q\|_{L^{2}}^{2},
\end{align*}
which are all acceptable bounds. For $\mathsf{err}_{4}$, we \eqref{g1} and duality to estimate the product: 
\begin{align*}
\| q' \g(\vk)\|_{H^{-1}} & =\sup_{\|h\|_{H^1}=1} \bigg\vert \int q'\g(\vk) h dx \bigg\vert  = \sup_{\|h\|_{H^1}=1}\| q\|_{L^2}\| h \g(\vk)\|_{H^{1}}  \les \|q\|_{L^{2}}^{3},
\end{align*}
which yields the good bound
\begin{align*}
\| \mathsf{err}_{4}\|_{H^{-2}} & \les \kk^{-2}( \|q\|_{L^2}+\|q\|_{L^2}^{3}).
\end{align*}
For $\mathsf{err}_{5}$, $\mathsf{err}_{7}$ and $\mathsf{err}_{8}$, we rely on equicontinuity to obtain acceptable bounds. Indeed, using $L^1 \subset H^{-1}$ and Proposition~\ref{PROP:gpr}, H\"{o}lder's inequality and Sobolev embedding, we have
\begin{align*}
\| \mathsf{err}_{5}\|_{H^{-2}}&\les \big\| g_{+}(\vk) q\big(\tfrac{(-\dd^2)}{4\kk^2-\dd^2}q \big) \big\|_{H^{-1}} \les \|g_{+}(\vk)\|_{L^{\infty}}\|q\|_{L^2} \big\| \tfrac{(-\dd^2)}{4\kk^2-\dd^2}q \big\|_{L^2}  \les \|q\|_{L^2}^{2}\big\| \tfrac{(-\dd^2)}{4\kk^2-\dd^2}q \big\|_{L^2}, \\
\| \mathsf{err}_{8}\|_{H^{-2}}&\les \|g_{+}(\vk)\|_{L^{\infty}}\big\| \dd (2\kk-\dd)^{-1}q \|_{L^2}  \| \dd (2\kk+\dd)^{-1}q\|_{L^2},
 \end{align*}
where the right hand side of these estimates tend to zero as $\kk \to \infty$, uniformly in $q\in Q_{\ast}$ and $|t|\leq T$, by Lemma~\ref{LEM:equ2} and \eqref{tightness2}. Again, by duality, we obtain 
\begin{align*}
\big\| \g(\vk)  \dd^3 (4\kk^2-\dd^2)^{-1}q \big\|_{H^{-1}} & \les \|\g(\vk)\|_{H^{1}} \big\|  \dd^2 (4\kk^2-\dd^2)^{-1}q\big\|_{L^2},
\end{align*}
and thus
\begin{align*}
\| \mathsf{err}_{7}\|_{H^{-2}}  & \les \|q\|_{L^{2}}^{2}  \big\| \dd^{2} (4\kk^2 -\dd^2)^{-1}q\big\|_{L^2},
\end{align*}
which tends to zero uniformly in $q\in Q_{\ast}$ as $\kk \to \infty$ and $|t|\leq T$ by Lemma~\ref{LEM:equ2}.

It remains to estimate $\mathsf{err}_{6}$.
We write 
\begin{align*}
\mathsf{err}_{6}&= \tfrac{8\vk \kk^4}{\kk^2-\vk^2}(\g(\vk)+1)g_{+}^{[ 3]}(\kk)+\tfrac{8\vk \kk^4}{\kk^2-\vk^2}(\g(\vk)+1)g_{+}^{[\geq 5]}(\kk)  =: \mathsf{err}_{6,1}+\mathsf{err}_{6,2}.
\end{align*}
From \eqref{gpbig} and Proposition~\ref{PROP:gpr},
\begin{align*}
\|\mathsf{err}_{6,2}\|_{H^{-2}} \les \kk^{2} \| \g(\vk)g_{+}^{[\geq 5]}(\kk)\|_{L^2}\les \kk^{-1} \|q\|_{L^2}^{7}.
\end{align*}
For $\mathsf{err}_{6,1}$, we start with some preliminary estimates. From \eqref{g2term} and \eqref{pterms}, we have
\begin{align*}
\|g_{+}^{[3]}(\kk) \|_{H^{-2}} &\les  \big\|(4\kk^2-\dd^2)^{-1}\big(  q \g^{[2]}(\kk)\big) \bigg\|_{H^{-1}}  \les \kk^{-2}\|q\|_{L^2} \| \g^{[2]}(\kk)\|_{L^{\infty}} \les \kk^{-3}\|q\|_{L^2}^{3}.
\end{align*}
Similarly, using \eqref{pterms} as well as arguing by duality, yields
\begin{align*}
\|\g(\vk)g_{+}^{[3]}(\kk) \|_{H^{-1}}  & \les \kk^{-3} \|q\|_{L^{2}}^{5}
\end{align*}
and thus
\begin{align*}
\|\mathsf{err}_{6,1}\|_{H^{-2}} \les \kk^{-1} (\|q\|_{L^2}^{3}+\|q\|_{L^{2}}^{5}).
\end{align*}
Thus, each individual term of \eqref{rexpansion} tends to zero in $H^{-2}$, uniformly in $q\in Q_{\ast}$ and $|t|\leq T$ as $\kk\to \infty$. This completes the proof of \eqref{H-2limit2} and hence also the proof of \eqref{H-2limit}.
\end{proof}

\begin{proof}[Proof of Theorem~\ref{THM:GWP}]
We prove the following statement: given $T>0$ and $q_{n}(0) \in \mathcal{S}$ converging in $L^{2}$, then the  corresponding sequence of solutions $q_{n}(t)$ in $\mathcal{S}$ to \eqref{mkdv} converge in $L^2$, uniformly in $|t|\leq T$. 

To this end, let $\dl_0>0$ denote the smallest $\dl>0$ amongst all the claims in this paper thus far. Then, we choose $\kk \geq 1$ such that $$\kk^{-\frac{1}{2}} \sup_{n\in \NB}\|q_{n}(0)\|_{L^2} \leq \dl_0.$$
Namely, we know have that $q_{n}(0)\in B_{\dl_0, \kk}$, with $\kk$ independent of $n\in \mathbb{N}$.

Let $Q=\{q_{n}(0)\}$ and define $Q_{\ast}$ as in \eqref{Qast}, with $H_1=H_{\text{mKdV}}$ and $H_2=H_{\kk}$. Note that by Proposition~\ref{PROP:equicty}, $Q^{\ast}$ is equicontinuous in $L^2$.
By the commutativity of the flows, we have 
\begin{align*}
q_{n}(t)=e^{tJ\nabla( H_{\text{mKdV}}-H_{\kk})}\circ e^{tJ\nabla H_{\kk}}q_{n}(0).
\end{align*}
Hence,
\begin{align}
\sup_{|t|\leq T} \|q_{n}(t)-q_{m}(t)\|_{L^2}& \leq \sup_{|t|\leq T} \|e^{tJ\nabla H_{\kk}}q_{n}(0) -e^{tJ\nabla H_{\kk}}q_{m}(0) \|_{L^2} \notag \\
& \hphantom{XX}+ 2\sup_{q\in Q_{\ast}}\sup_{|t|\leq T} \|e^{tJ\nabla( H_{\text{mKdV}}-H_{\kk})}q-q\|_{L^2}. \label{L2diff1}
\end{align}
By Proposition~\ref{PROP:GWPHk}, the well-posedness of the $H_{\kk}$-flow in $L^2$ implies that the first term in \eqref{L2diff1} tends to zero as $n,m\to \infty$. To deal with the second term, we perform a change of variables: fix $\vk\geq 4$ and let 
$q(t):= e^{tJ\nabla (H_{\text{mKdV}}-H_{\kk})}q$
for $q\in Q^{\ast}$ and we increase $\kk$ if necessary so that $\kk\geq 2\vk$. Note that $q\in (Q_{\ast})_{\ast}$ for any $t\in \R$. 
We use the change of variables 
\begin{align}
q(t) \mapsto \tfrac{\mu}{4\vk} g_{-}(\vk; q(t)). \label{cov}
\end{align}
By Proposition~\ref{PROP:diffflow}, we have 
\begin{align*}
\lim_{\kk\to \infty} \sup_{q\in Q_{\ast}} \sup_{|t|\leq T} \| g_{-}(\vk; q(t))-g_{-}(\vk;q)\|_{H^{2}}=0.
\end{align*}
The diffeomorphism property of the change of variables \eqref{cov} then implies that \eqref{L2diff1} vanishes as $\kk \to \infty$. The remaining claims in Theorem~\ref{THM:GWP} follow exactly as in \cite[Corollary 5.2 and 5.3] {KV}, with well-posedness in $H^{s}$, for $0<s<\tfrac{1}{2}$ employing the bounds and equicontinuity proved in Proposition~\ref{PROP:equicty}.
\end{proof}

\begin{remark}\rm  \label{RMK:H16}
In the higher regularity setting of $H^{s}$, for $s\geq \tfrac{1}{6}$, we do not need to apply the change of variables \eqref{cov} and we can replace the approximation property in \eqref{H2limit} with 
\begin{align}
\lim_{\kk\to \infty}\sup_{q\in Q}\sup_{|t|\leq T}   \| e^{tJ\nabla(H_{\textup{mKdV}}-H_{\kk})}q-q\|_{H^{\frac{1}{6}}}=0. \label{H2limit2}
\end{align}
Namely, we establish \textit{directly} that the $H_{\kk}$ flow approximates the $H_{\text{mKdV}}$ flow. This is possible as the additional spatial regularity in $H^{\frac{1}{6}}$ implies that $q^3$ is a well-defined distribution. The change of variables \eqref{cov} is chosen because, apart from its diffeomorphism property, it evolves simply under the linearised $H_{\text{mKdV}}$-flow, and the nonlinear terms in \eqref{rHmkdvflow} are well-defined in the sense of distributions under the weaker assumption of merely $q\in L^2$.
\end{remark}

\begin{ackno}\rm The author would like to thank Rowan Killip and Monica Vi\c{s}an for their support and for helpful discussions.
\end{ackno}


\begin{thebibliography}{99}	

\bibitem{ABFS}
 L.~Abdelouhab, J.~L.~Bona, M.~Felland, J.~C.~Saut, {\it Nonlocal models for nonlinear, dispersive waves}, Phys. D 40(3), 360--392 (1989).
 
 

\bibitem{BS1}
J.~L.~Bona, R.~Smith, {\it The initial-value problem for the Korteweg-de Vries equation}, Philos.
Trans. Roy. Soc. London Ser. A 278(1287), 555--601 (1975). 


\bibitem{BO2}
J.~Bourgain, 
{\it Fourier transform restriction phenomena for certain lattice subsets and applications to
nonlinear evolution equations. II. The KdV equation,} Geom. Funct. Anal. 3(3), 209--262 (1993).

\bibitem{BO2}
J.~Bourgain, {\it Periodic Korteweg-de Vries equation with measures as initial data,} Selecta Math. (N.S.)
3(2), 115--159 (1997)

\bibitem{BKV}
B.~Bringmann, R.~Killip, M.~Vi\c{s}an,
{\it Global well-posedness for the fifth-order KdV equation in $H^{-1}(\R)$,}
Ann. PDE 7 (2021), no. 2, Paper No. 21, 46 pp.
35Q53.

\bibitem{Chapouto1}
A.~Chapouto, {\it A remark on the well-posedness of the modified KdV equation in the Fourier–Lebesgue
spaces}, Discrete Contin. Dyn. Syst. 41(8), 3915--3950 (2021).

\bibitem{Chapouto2}
A.~Chapouto, {\it A Refined Well-Posedness Result for the Modified KdV Equation in the Fourier–Lebesgue Spaces},  J. Dynam. Differential Equations 35 (2023), no.3, 2537--2578.

\bibitem{CCT1}
 M.~Christ, J.~Colliander, T.~Tao, {\it Asymptotics, frequency modulation, and low regularity ill-posedness
for canonical defocusing equations}, Amer. J. Math. 125 (2003), no. 6, 1235--1293.

\bibitem{CCT2}
M.~Christ, J.~Colliander, T.~Tao, {\it Ill-posedness for nonlinear Schr\"{o}odinger and wave equations,} Preprint, 2003. arXiv:math/0311048.

\bibitem{CHT}
 M.~Christ, J.~Holmer, D.~Tataru, {\it Low regularity a priori bounds for the modified Korteweg-de Vries
equation}, Lib. Math. (N.S.) 32 (2012), no. 1, 51--75.

\bibitem{CKSTT1}
J.~Colliander, M.~Keel, G.~Staffilani, H.~Takaoka, T.~Tao, {\it Sharp global well-posedness for KdV and
modified KdV on $\R$ and $\T$}, J. Amer. Math. Soc. 16 (2003), no. 3, 705--749.



\bibitem{Guo}
Z. Guo, {\it Global well-posedness of Korteweg-de Vries equation in $H^{-\frac 34}(\R)$}, J. Math. Pures Appl. (9), 91(6): 583--597, 2009.


\bibitem{HGKV}
B.~Harrop-Griffiths, R.~Killip, M.~Vi\c{s}an,
{\it Sharp well-posedness for the cubic NLS and mKdV in $H^{s}(\R)$},
Forum of Mathematics, Pi. Volume 12, 2024, e6.

\bibitem{HGKV2}
B.~Harrop-Griffiths, R.~Killip, M.~Vi\c{s}an,
{\it Large-data equicontinuity for the derivative NLS},
Int. Math. Res. Not. rnab374, https://doi.org/10.1093/imrn/rnab374


\bibitem{HGKNV}
B.~Harrop-Griffiths, R.~Killip, M.~Ntekoume, M.~Vi\c{s}an,
{\it Global well-posedness for the derivative nonlinear Schrödinger equation in $L^2(\R)$},
to appear in J. Eur. Math. Soc.


\bibitem{Kato2}
T.~Kato, 
{\it On the Korteweg-de Vries equation}, Manuscripta Math. 28 (1979), no. 1-3, 89--99.

\bibitem{Kato1}
T.~Kato, 
{\it On the Cauchy problem for the (generalized) Korteweg–de Vries equation}, In Studies in applied mathematics, volume 8 of Adv. Math. Suppl. Stud., pages 93--128. Academic Press, New York, 1983.

\bibitem{KapMol}
T.~Kappeler, J.-C.~Molnar, {\it On the well-posedness of the defocusing mKdV equation below $L^2$}, SIAM
J. Math. Anal. 49(3), 2191--2219 (2017).

\bibitem{KapTopKdv}
T.~Kappeler, P.~Topalov, {\it Global wellposedness of KdV in $H^{-1}(\T,\R)$}, Duke Math. J. 135 no. 2 (2006), 327--360.

\bibitem{KapTop}
T.~Kappeler, P.~Topalov, {\it Global well-posedness of mKdV in $L^2(\T,\R)$}, Commun. Partial Differ. Equ.
30(1–3), 435--449 (2005).




\bibitem{KPV1}
C.~Kenig, G.~Ponce, L.~Vega, {\it On the (generalized) Korteweg-de Vries equation}, Duke Math. J. 59
(1989), no. 3, 585--610.

\bibitem{KPV2}
C.~Kenig, G.~Ponce, L.~Vega, {\it Well-posedness and scattering results for the generalized Korteweg-de
Vries equation via the contraction principle}, Comm. Pure Appl. Math. 46 (1993), 527--620.

\bibitem{KPV3}
 C.~Kenig, G.~Ponce, L.~Vega, {\it On the ill-posedness of some canonical dispersive equations}, Duke Math.
J. 106 (2001), no. 3, 617--633.


\bibitem{KNV}
R.~Killip, M.~Ntekoume, M.~Vi\c{s}an, {\it  On the well-posedness problem for the derivative
nonlinear Schr\"{o}dinger equation}, Anal. PDE 16 (2023), no. 5, 1245--1270.



\bibitem{KV}
R.~Killip, M.~Vi\c{s}an,
{\it KdV is well-posed in $H^{-1}$},
Ann. of Math. (2), 190(1):249–305, 2019.


\bibitem{KVZ}
R.~Killip, M.~Vi\c{s}an, X.~Zhang, 
{\it Low regularity conservation laws for integrable PDE},
 Geom. Funct. Anal. 28 (2018), no. 4, 1062--1090. 


\bibitem{Kishimoto}
N.~Kishimoto, {\it Well-posedness of the Cauchy problem for the Korteweg-de Vries equation at the critical
regularity}, Differential Integral Equations 22 (2009), no. 5-6, 447--464.

\bibitem{Kishimoto2}
N.~Kishimoto, {\it A remark on norm inflation for nonlinear Schr\"{o}dinger equations}, Commun. Pure Appl. Anal., 18(3):1375--1402, 2019.


\bibitem{KT}
H.~Koch and D.~Tataru, {\it Conserved energies for the cubic nonlinear Schr\"{o}dinger equation in one dimension}, Duke Math. J. 167 no. 17 (2018), 3207--3313.


\bibitem{KOY}
S.~Kwon, T.~Oh, H.~Yoon, {\it Normal form approach to unconditional well-posedness of nonlinear dispersive
PDEs on the real line}, Ann. Fac. Sci. Toulouse Math. 29 (2020), no. 3, 649--720.





\bibitem{Miura}
R.~M.~Miura,
{\it  Korteweg-de Vries equation and generalizations. I. A remarkable explicit non- linear transformation}, J. Mathematical Phys., 9:1202--1204, 1968.

\bibitem{Molinet}
L.~Molinet, {\it Sharp ill-posedness results for KdV and mKdV equations on the torus}, Adv. Math. 230(4–6),
1895--1930 (2012).

\bibitem{MPV1}
L.~Molinet, D.~Pilod, S.~Vento, {\it Unconditional uniqueness for the modified Korteweg-de Vries equation on the line}, Rev. Mat. Iberoam. 34 (2018), no. 4, 1563–1608.

\bibitem{MPV2}
L.~Molinet, D.~Pilod, S.~Vento, {\it On unconditional well-posedness for the periodic modified Korteweg-de
Vries equation,} J. Math. Soc. Japan 71(1), 147--201 (2019).

\bibitem{NTT}
K.~Nakanishi, H.~Takaoka, Y.~Tsutsumi, {\it Local well-posedness in low regularity of the mKdV equation
with periodic boundary condition,} Discrete Contin. Dyn. Syst. 28(4), 1635--1654 (2010).

%\bibitem{Ngyuen}
%T.~Nguyen, {\it Power series solution for the modified KdV equation}, Electron. J. Differ. Equ. 71, 1--10 (2008).



\bibitem{Oh1}
 T.~Oh, {\it A remark on norm inflation with general initial data for the cubic nonlinear Schr\"{o}dinger equations in negative Sobolev spaces}, Funkcial. Ekvac., 60(2):259--277, 2017.
 
 \bibitem{Pego}
 R. L. Pego, {\it Compactness in $L^2$ and the Fourier transform,} Proc. Amer. Math. Soc., 95(2):252--254, 1985.
 
 \bibitem{Schippa}
R.~Schippa, {\it On the existence of periodic solutions to the modified Korteweg-de Vries equation below $H^{\frac{1}{2}}(\T)$}, J. Evol. Equ. 20(3), 725--776 (2020).
 
 \bibitem{TT}
 H.~Takaoka, Y.~Tsutsumi, {\it Well-posedness of the Cauchy problem for the modified KdV equation with
periodic boundary condition}, Int. Math. Res. Not. 56, 3009--3040 (2004).

\bibitem{Tao}
T.~Tao, {\it Multilinear weighted convolution of $L^2$-functions, and applications to nonlinear dispersive equations}, Amer. J. Math. 123 (2001), no. 5, 839--908.



\bibitem{MTsut}
M.~Tsutsumi, 
{\it Weighted Sobolev spaces and rapidly decreasing solutions of some nonlinear dispersive wave equations}, J. Differential Equations, 42(2):260--281, 1981.


	
%



\end{thebibliography}
\end{document}